\documentclass{gtpart}
\usepackage{pinlabel}
\usepackage{graphicx}
\title[Injective homomorphisms of mapping class groups of non-orientable surfaces]{Injective homomorphisms of mapping class groups of non-orientable surfaces}
\author[E Irmak]{Elmas Irmak}
\givenname{Elmas}
\surname{Irmak}
\address{Bowling Green State University, Department of Mathematics and Statistics, Bowling Green, 43403, OH, USA}
\email{eirmak@bgsu.edu}
\urladdr{}
\author[L Paris]{Luis Paris}
\givenname{Luis}
\surname{Paris}
\address{IMB, UMR 5584,CNRS, Univ. Bourgogne Franche-Comté, 21000 Dijon, France}
\email{lparis@u-bourgogne.fr}
\urladdr{}


\volumenumber{}
\issuenumber{}
\publicationyear{}
\papernumber{}
\startpage{}
\endpage{}
\doi{}
\MR{}
\Zbl{}
\received{}
\revised{}
\accepted{}
\published{}
\publishedonline{}
\proposed{}
\seconded{}
\corresponding{}
\editor{}
\version{}
\newtheorem{thm}{Theorem}[section]
\newtheorem{lem}[thm]{Lemma}
\newtheorem{prop}[thm]{Proposition}
\newtheorem{corl}[thm]{Corollary}

\theoremstyle{definition}
\newtheorem*{rem}{Remark}

\numberwithin{equation}{section}

\makeatletter
\renewcommand{\thefigure}{\ifnum \c@section>\z@ \thesection.\fi
 \@arabic\c@figure}
\@addtoreset{figure}{section}
\makeatother

\begin{document}

\def\Homeo{{\rm Homeo}} \def\MM{\mathcal M} \def\GG{\mathcal G}
\def\Aut{{\rm Aut}} \def\Out{{\rm Out}} \def\Com{{\rm Com}}
\def\TT{\mathcal T} \def\CC{\mathcal C} \def\AA{\mathcal A}
\def\S{\mathbb S} \def\Z{\mathbb Z} \def\SS{\mathcal S}
\def\rk{{\rm rk}} \def\OO{\mathcal O} \def\SSS{\mathfrak S}
\def\Ker{{\rm Ker}} \def\BB{\mathcal B} \def\PP{\mathcal P}
\def\id{{\rm id}}


\begin{abstract}
Let $N$ be 
a compact, connected, non-orientable 
surface of genus $\rho$ with $n$ boundary components, with $\rho \ge 5$ and $n \ge 0$, and let $\MM (N)$ be the mapping class group of $N$.
We show that, if $\GG$ is a finite index subgroup of $\MM (N)$ and $\varphi: \GG \to \MM (N)$ is an injective homomorphism, then there exists $f_0 \in \MM (N)$ such that $\varphi (g) = f_0 g f_0^{-1}$ for all $g \in \GG$.
We deduce that the abstract commensurator of $\MM (N)$ coincides with $\MM (N)$.
\end{abstract}

\maketitle


\section{Introduction}

Let $N$ be a connected and compact surface that can be non-orientable and have boundary but whose Euler characteristic is negative.
We denote by $\Homeo (N)$ the group of homeomorphisms of $N$.
The \emph{mapping class group} of $N$, denoted by $\MM (N)$, is the group of isotopy classes of elements of $\Homeo (N)$.
Note that it is usually assumed that the elements of $\Homeo (N)$ preserve the orientation when $N$ is orientable.
However, the present paper deals with non-orientable surfaces and we want to keep the same definition for all surfaces, so, for us, the elements of $\Homeo (N)$ can reverse the orientation even if $N$ is orientable.
On the other hand, we do not assume that the elements of $\Homeo (N)$ pointwise fix the boundary.
In particular, a Dehn twist along a circle isotopic to a boundary component is trivial.
If $N$ is not connected, the \emph{mapping class group} of $N$, denoted by $\MM (N)$, has the same definition.
Although our results concern connected surfaces, we shall use in some places mapping class groups of non-connected surfaces.

The purpose of this paper is to prove the following.

\begin{thm}\label{thm1_1}
Let $N$ be a 
compact, connected, non-orientable 
surface of genus $\rho$ with $n$ boundary components, with $\rho \ge 5$ and $n \ge 0$, let $\GG$ be a finite index subgroup of $\MM (N)$, and let $\varphi: \GG \to \MM (N)$ be an injective homomorphism.
Then there exists $f_0 \in \MM (N) $ such that $\varphi (g) = f_0 g f_0^{- 1}$ for all $ g \in \GG$.
\end{thm}

The same result is known for orientable surfaces (see Irmak \cite{Irmak2,Irmak3,Irmak4}, Behrstock--Margalit \cite{BehMar1}, Bell--Margalit \cite{BelMar1}).
A first consequence of Theorem \ref{thm1_1}  is as follows.
This corollary was proved in Atalan--Szepietowski \cite{Atala1,AtaSze1} for compact, connected, non-orientable surfaces of genus $\rho \ge 7$.
Our contribution is just an extension to compact, connected, non-orientable surfaces of genus $5$ and $6$.

\begin{corl}\label{corl1_2}
Let $N$ be a compact, connected, non-orientable surface of genus $\rho$ with $n$ boundary components, with $\rho \ge 5$ and $n \ge 0$.
Then $ \Aut (\MM (N)) \cong \MM (N)$ and $\Out (\MM (N)) \cong \{1 \}$.
\end{corl}

Note that the ideas of Atalan--Szepietowski \cite{Atala1, AtaSze1} cannot be extended to the more flexible framework of Theorem \ref{thm1_1} because they use rather precise relations in $\MM (N)$ that do not necessarily hold in a finite index subgroup.

Let $G$ be a group.
We denote by $\widetilde{\Com} (G)$ the set of triples $(A, B, \varphi)$, where $A$ and $B$ are finite index subgroups of $G$ and $\varphi: A \to B$ is an isomorphism.
Let $\sim$ be the equivalence relation on $\widetilde {\Com} (G)$ defined as follows.
We set $(A_1, B_1, \varphi_1) \sim (A_2, B_2, \varphi_2)$ if there exists a finite index subgroup $C$ of $A_1 \cap A_2$ such that $\varphi_1(g) = \varphi_2 (g)$ for all $g \in C$.
The (abstract) \emph{commensurator} of $G$ is defined to be the quotient $\Com (G) = \widetilde{\Com} (G) / \sim$.
This is a group.
A more general consequence of Theorem \ref{thm1_1} is the following.

\begin{corl}\label{corl1_3}
Let $N$ be a 
compact, connected, non-orientable
surface of genus $\rho$ with $n$ boundary components, with $\rho \ge 5$ and $n \ge 0$.
Then $\Com (\MM (N)) \cong \MM (N)$.
\end{corl}

Let $N$ be a 
compact, connected, non-orientable
surface of genus $\rho$ with $n$ boundary components, with $\rho \ge 5$ and $n \ge 0$.
We denote by $\CC (N)$ the complex of curves of $N$ and by $\TT(N)$ the subcomplex of $\CC (N)$ formed by the isotopy classes of two-sided circles.
For $\alpha, \beta \in \CC (N) $ we denote by $i (\alpha, \beta)$ the intersection index of $\alpha$ and $\beta$ and, for $\alpha \in \TT (N)$, we denote $ t_\alpha$ the Dehn twist along $\alpha$.
A map $\lambda : \TT (N) \to \TT (N)$ is called a \emph{super-injective simplicial map if the condition $i( \lambda (\alpha), \lambda (\beta)) = 0$ is equivalent to the condition $i (\alpha, \beta) = 0$ for all $\alpha, \beta \in \TT (N)$.}

The proof of Theorem \ref{thm1_1} is based on two other theorems.
The first one, proved in Irmak--Paris \cite[Theorem 1.1]{IrmPar1}, says that, if $\lambda: \TT (N) \to \TT (N)$ is a super-injective simplicial map, then there exists $f_0 \in \MM (N)$ such that $\lambda (\alpha) = f_0 (\alpha) $ for all $\alpha \in \TT (N)$.
Let $\TT'(N)$ be the subfamily of $\TT (N)$ formed by the isotopy classes of circles that do not bound one-holed Klein bottles.
The second theorem is the aim of Section \ref{Sec3} of this paper (see Theorem \ref{thm3_1}).
It says that, if $\GG$ is a finite index subgroup of $\MM (N) $ and $\varphi: \GG \to \MM (N)$ is an injective homomorphism, then there exists a super-injective simplicial map $\lambda: \TT (N) \to \TT (N) $ which sends a 
non-trivial
power of $t_\alpha$ to a non-trivial power of $t_{\lambda (\alpha)}$ for all $\alpha \in \TT'(N)$.

The proof of Theorem \ref{thm3_1} uses several results on reduction classes of elements of $\MM (N)$ and on  maximal abelian subgroups of $\MM (N)$.
These results are classical and widely used in the theory of mapping class groups of orientable surfaces, but they are more or less incomplete in the literature for mapping class groups of non-orientable surfaces.
So, the aim of Section \ref{Sec2} is to recall and complete what is known 
on non-orientable case.

In order to prove Theorem \ref{thm3_1} we first prove an algebraic characterization of some Dehn twists up to roots and powers (see Proposition \ref{prop3_3}).
This characterization seems interesting by itself and we believe that it could be used for other purposes.
Note also that this characterization is independent from related works of Atalan \cite{Atala1,Atala2} and Atalan--Szepietowski \cite{AtaSze1}.

The final proofs of Theorem \ref{thm1_1} and Corollary \ref{corl1_3} are given in Section \ref{Sec4}. 
The proof of Corollary \ref{corl1_2} is left to the reader.


\section{Preliminaries}\label{Sec2}

\subsection{Canonical reduction systems}

From now on $N=N_{\rho,n}$ denotes 
compact, connected, non-orientable
surface of genus $\rho \ge 5$ with $n \ge 0$ boundary components. 
A \emph{circle} in $N$ is an embedding $a : \S^1 \hookrightarrow N\setminus \partial N$.
This is assumed to be non-oriented. 
It is called \emph{generic} if it does not bound a disk, it does not bound a M\"obius band, and it is not isotopic to a boundary component.  
The isotopy class of a circle $a$ is denoted by $[a]$.
A circle $a$ is called \emph{two-sided} (resp. \emph{one-sided}) if a regular neighborhood of $a$ is an annulus (resp. a M\"obius band).
We denote by $\CC(N)$ the set of isotopy classes of generic circles, and by $\TT (N)$ the subset of $\CC (N)$ of isotopy classes of generic two-sided circles.

The \emph{intersection index} of two classes $\alpha, \beta \in \CC(N)$ is 
$i(\alpha, \beta) = \min\{|a \cap b| \mid a \in \alpha \text{ and } b \in \beta\}$.
The set $\CC(N)$ is endowed with a structure of simplicial complex, where a finite set $\AA$ is a simplex if $i(\alpha, \beta) = 0$ for all $\alpha, \beta \in \AA$.
This simplicial complex is called the \emph{curve complex} of $N$.
Note that $\MM(N)$ acts naturally on $\CC(N)$, and this action is a simplicial action. 
Note also that $\TT(N)$ is invariant under this action.

Let $\AA=\{\alpha_1, \dots, \alpha_r\}$ be a simplex of $\CC(N)$.
Let $\MM_\AA(N)$ denote the stabilizer of $\AA$ in $\MM(N)$.
Choose pairwise disjoint representatives  $a_i \in \alpha_i$, $i \in \{1, \dots, r\}$, and denote by $N_\AA$ the natural compactification of $N \setminus (\cup_{i=1}^r a_i)$.
The groups $\MM_\AA(N)$ and $\MM(N_\AA)$ are linked by a homomorphism $\Lambda_\AA : \MM_\AA(N) \to \MM(N_\AA)$, called the \emph{reduction homomorphism} along $\AA$, defined as follows.
Let $f \in \MM_\AA(N)$.
Choose a representative $F \in \Homeo(N)$ of $f$ such that $F(\{a_1, \dots, a_r\}) = \{a_1, \dots, a_r\}$.
The restriction of $F$ to $N \setminus ( \cup_{i=1}^r a_i)$ extends in a unique way to a homeomorphism $\hat F \in \Homeo(N_\AA)$.
Then $\Lambda_\AA (f)$ is the element of $\MM (N_\AA)$ represented by $\hat F$.

We denote by $t_\alpha$ the Dehn twist along a class $\alpha \in \TT (N)$.
This is defined up to a power of $\pm 1$, since its definition depends on the choice of an orientation in a regular neighborhhood of a representative of $\alpha$, but this lack of precision will not affect the rest of the paper.  
If $\AA$ is a simplex of $\CC (N)$, we set $\AA_\TT = \AA \cap \TT(N)$, and we denote by $Z_\AA$ the subgroup of $\MM(N)$ generated by $\{ t_\alpha \mid \alpha \in \AA_\TT \}$.
The following is a classical result for mapping class groups of orientable surfaces.
It is due to Stukow \cite{Stuko1,Stuko2} for non-orientable surfaces. 

\begin{prop}[Stukow \cite{Stuko1,Stuko2}]\label{prop2_1}
Let $\AA$ be a simplex of $\CC(N)$.
Then $Z_\AA$ is the kernel of $\Lambda_\AA$ and it is a free abelian group of rank $|\AA_\TT|$.
\end{prop}

Let $f \in \MM(N)$.
We say that $f$ is \emph{pseudo-Anosov} if $\CC(N) \neq \emptyset$ and $f^n(\alpha) \neq \alpha$ for all $\alpha \in \CC(N)$ and all $n \in \Z \setminus \{ 0\}$.
We say that $f$ is \emph{periodic} if it is of finite order.
We say that $f$ is \emph{reducible} if there is a non-empty simplex $\AA$ of $\CC(N)$ such that $f(\AA) = \AA$.
In that case, $\AA$ is called a \emph{reduction system} for $f$ and an element of $\AA$ is called a \emph{reduction class} for $f$.
An element $f \in \MM(N)$ can be both periodic and reducible, but it cannot be both pseudo-Anosov and periodic, and it cannot be both pseudo-Anosov and reducible.

Let $f \in \MM(N)$ be a reducible element and let $\AA$ be a reduction system for $f$.
Let $N_1, \dots, N_\ell$ be the connected components of $N_\AA$.
Choose $n \in \Z \setminus \{ 0 \}$ such that $f^n (N_i)= N_i$ for all $i$ and denote by $\Lambda_{i,\AA}(f^n) \in \MM(N_i)$ the restriction of $\Lambda_\AA (f^n)$ to $N_i$.
We say that $\AA$ is an \emph{adequate reduction system} for $f$ if, for all $i \in \{1, \dots, \ell \}$, $\Lambda_{i,\AA}(f^n)$ is either periodic or pseudo-Anosov. 
This definition does not depend on the choice of $n$.
The following was proved by Thurston \cite{FaLaPo1} for orientable surfaces and then was extended to non-orientable surfaces by Wu \cite{Wu1}.

\begin{thm}[Thurston \cite{FaLaPo1}, Wu \cite{Wu1}]\label{thm2_2}
A mapping class $f \in \MM(N)$ is either pseudo-Anosov, periodic or reducible. 
Moreover, if it is reducible, then it admits an adequate reduction system.
\end{thm}

Let $f \in \MM(N)$.
A reduction class $\alpha$ for $f$ is an \emph{essential reduction class} for $f$ if for each $\beta \in \CC (N)$ such that $i(\alpha, \beta) \neq 0$ and for each integer $m \in \Z \setminus \{0\}$, the classes $f^m (\beta)$ and $\beta$ are distinct.
We denote by $\SS (f)$ the set of essential reduction classes for $f$. 
Note that $\SS(f) = \emptyset$ if and only if $f$ is either pseudo-Anosov or periodic.
The following was proved by Birman--Lubotzky--McCarthy \cite{BiLuMc1} for orientable surfaces and by Wu \cite{Wu1} and Kuno \cite{Kuno1} for non-orientable surfaces.

\begin{thm}[Wu \cite{Wu1}, Kuno \cite{Kuno1}]\label{thm2_3}
Let $f \in \MM(N)$ be a non-periodic reducible mapping class.  
Then $\SS(f) \neq \emptyset$, $\SS (f)$ is an adequate reduction system for $f$, and every adequate reduction system for $f$ contains $\SS (f)$.
In other words, $\SS (f)$ is the unique minimal adequate reduction system for $f$.
\end{thm}

\begin{rem}
The proof of Theorem \ref{thm2_3} for orientable surfaces given in Birman--Lubotzky--McCarthy \cite{BiLuMc1} does not extend to non-orientable surfaces, since Lemma 2.4 in \cite{BiLuMc1} is false for non-orientable surfaces and this lemma is crucial in the proof. 
The proof for non-orientable surfaces is partially made in Wu \cite{Wu1} using oriented double coverings, and it is easy to get the whole result by these means. 
Kuno's approach \cite{Kuno1} is different in the sense that she replaces Lemma 2.4 of \cite{BiLuMc1} by another lemma of the same type. 
\end{rem}

If $f \in \MM (N)$ is a non-periodic reducible mapping class, then the set $\SS(f)$ is called the \emph{essential reduction system} for $f$.
Recall that, if $f$ is either periodic or pseudo-Anosov, then $\SS (f) = \emptyset$.
The following lemma follows from the definition of $\SS(f)$.

\begin{lem}\label{lem2_4}
Let $f \in \MM(N)$.
Then
\begin{itemize}
\item[(1)]
$\SS(f^n) = \SS(f)$ for all $n \in \Z\setminus \{0\}$,
\item[(2)]
$\SS(gfg^{-1}) = g(\SS(f))$ for all $g \in \MM(N)$.
\end{itemize}
\end{lem}

Moreover, we will often use the following.

\begin{lem}\label{lem2_5}
Let $\AA=\{\alpha_1, \dots, \alpha_p\}$ be a simplex of $\TT(N)$, and, for every $i \in \{1, \dots, p\}$, let $k_i \in \Z \setminus \{0\}$.
Set $g = t_{\alpha_1}^{k_1} \cdots t_{\alpha_p}^{k_p}$.
Then $\SS(g) = \AA$.
\end{lem}

\begin{proof}
We have $g(\alpha_i) = \alpha_i$ for all $i \in \{1, \dots, p\}$, hence $\AA$ is a reduction system for $g$.
The reduction $\Lambda_\AA(g)$ of $g$ along $\AA$ is the identity, hence $\AA$ is an adequate reduction system. 
Let $i \in \{1, \dots, p\}$.
Set $\AA_i = \AA \setminus \{\alpha_i\}$.
Note that $\AA_i$ is also a reduction system for $g$.
We choose disjoint representatives $a_j \in \alpha_j$, $j \in \{1, \dots, p\}$, we denote by $N_{\AA_i}$ the natural compactification of $N \setminus ( \cup_{j \neq i} a_j)$, and we denote by $N'$ the component of $N_{\AA_i}$ that contains $a_i$.
The restriction of $\Lambda_{\AA_i}(g)$ to $N'$ is $t_{a_i}^{k_i}$, which is neither periodic, nor pseudo-Anosov. 
Hence, $\AA_i$ is not an adequate reduction system.  
So, $\AA$ is the minimal adequate reduction system for $g$, that is, $\AA=\SS(g)$.
\end{proof}


\subsection{Abelian subgroups}

Let $\AA$ be a simplex of $\CC(N)$ and let $H$ be an abelian subgroup of $\MM(N)$.
We say that $\AA$ is a \emph{reduction system} for $H$ if $\AA$ is a reduction system for every element of $H$.
Similarly, we say that $\AA$ is an \emph{adequate reduction system} for $H$ if $\AA$ is an adequate reduction system for every element of $H$.
On the other hand, we set $\SS(H)=\bigcup_{f \in H} \SS(f)$.
The following can be proved exactly in the same way as Lemma 3.1 in Birman--Lubotzky--McCarthy \cite{BiLuMc1}.

\begin{lem}\label{lem2_6}
Let $H$ be an abelian subgroup of $\MM(N)$.
Then $\SS (H)$ is an adequate reduction system for $H$ and every adequate reduction system for $H$ contains $\SS(H)$.
\end{lem}

The \emph{rank} of $N$ is defined by $\rk (N) = \frac{3}{2} \rho + n -3$ if  $\rho$ is even and by $\rk (N) = \frac{1}{2} (3 \rho-1) + n -3$ if  $\rho$ is odd.
Let $\AA$ be a simplex of $\CC(N)$.
As ever, we set $\AA_\TT = \AA \cap \TT (N)$.
For $k \in \Z \setminus \{0 \}$, we denote by $Z_\AA [k]$ the subgroup of $\MM (N)$ generated by $\{ t_\alpha ^k \mid \alpha \in \AA_\TT \}$.
By Proposition \ref{prop2_1}, $Z_\AA [k]$ is a free abelian group of rank $|\AA_\TT|$.
The rank of an abelian group $H$ is denoted by $\rk (H)$.

Kuno \cite{Kuno1} proved that the maximal rank of an abelian subgroup of $\MM(N)$ is precisely $\rk (N)$.
Her proof is largely inspired by  Birman--Lubotzky--McCarthy \cite{BiLuMc1}.
In the present paper we need the following more precise statement whose proof is also largely inspired by Birman--Lubotzky--McCarthy \cite{BiLuMc1}.

\begin{prop}\label{prop2_7}
Let $H$ be an abelian subgroup of $\MM (N)$.
Then $\rk (H) \le \rk (N)$.
Moreover, if $\rk (H) = \rk (N)$, then the following claims hold.
\begin{itemize}
\item[(1)]
There exists a simplex $\AA$ in $\TT (N)$ such that $|\AA|=\rk (N)$ and $\SS (H) \cap \TT (N) \subset \AA$.
\item[(2)]
There exists $k \ge 1$ such that $Z_{\SS (H)} [k] \subset H$.
\item[(3)]
None of the connected components of $N_{\SS (H)}$ is homeomorphic to $N_{2,1}$.
\end{itemize}
\end{prop}

\begin{rem}
There are classes of circles that are not included in any simplex of $\TT(N)$ of cardinality $\rk (N)$.
These classes will be considered in Section \ref{Sec3}.
In particular, Part (1) of the above proposition is not immediate. 
\end{rem}

The rest of this subsection is dedicated to the proof of Proposition \ref{prop2_7}.

Let $\AA$ be a simplex of $\CC (N)$.
We say that $\AA$ is a \emph{pants decomposition} if $N_\AA$ is a disjoint union of pairs of pants. 
It is easily seen that any simplex of $\CC (N)$ is included in a pants decomposition.
The following lemma is partially proved in Irmak \cite[Lemma 2.2]{Irmak1}.

\begin{lem}\label{lem2_8}
Let $\AA$ be a simplex of $\CC (N)$.
Set $\AA_\TT = \AA \cap \TT (N)$ and $\AA_\OO = \AA \setminus \AA_\TT$.
Then 
\[
|\AA_\TT|  + \frac{1}{2} |\AA_\OO|\le \frac{3}{2} \rho + n -3\,.
\]
Moreover, equality holds if and only if $\AA$ is a pants decomposition.
\end{lem}

\begin{proof}
We choose a pants decomposition $\AA'$ which contains $\AA$.
Clearly, $|\AA'_\TT| \ge |\AA_\TT|$ and $|\AA'_\OO| \ge |\AA_\OO|$.
Let $p$ be the number of connected components of $N_{\AA'}$.
Since every connected component of $N_{\AA'}$ is a pair of pants, we have $2 |\AA'_\TT| + |\AA'_\OO| + n = 3p$ and $2 - \rho -n = \chi(N) = -p$.
These equalities imply that $|\AA_\TT'| + \frac{1}{2} |\AA_\OO'| =\frac{3}{2} \rho + n -3$.
It follows that $|\AA_\TT| + \frac{1}{2} |\AA_\OO| \le \frac{3}{2} \rho + n  -3$, and equality holds if and only if $\AA=\AA'$.
\end{proof}

\begin{corl}\label{corl2_9}
If $\AA$ is a simplex of $\TT (N)$, then $|\AA| \le \rk (N)$.
\end{corl}

\begin{proof}
We have $\AA_\TT = \AA$ and $\AA_\OO = \emptyset$.
By Lemma \ref{lem2_8}, it follows that $|\AA| \le \frac{3}{2} \rho +n -3$.
If $\rho$ is even, then the right hand side of this inequality is $\rk(N)$, hence $|\AA| \le \rk (N)$.
If $\rho$ is odd, then the right hand side of this inequality is not an integer, but the greatest integer less than the right hand side is $\rk (N) = \frac{1}{2} (3 \rho -1) + n -3$.
So, we have $|\AA| \le \rk (N)$ in this case, too. 
\end{proof}

\begin{rem}
There are simplices in $\TT (N)$ of cardinality $\rk (N)$. 
Their construction is left to the reader.
\end{rem}

For $\rho, n \ge 0$, we denote by $S_{\rho, n}$ the orientable surface of genus $\rho$ and $n$ boundary components. 
The first part of the following lemma is well-known.
The second part is proved in Scharlemann \cite{Schar1} and Stuko \cite{Stuko1}.
The third part is easy to prove.

\begin{lem}\label{lem2_10}
\begin{itemize}
\item[(1)]
The mapping class groups $\MM (S_{0,3})$ and $\MM (N_{1,2})$ are finite.
\item[(2)]
The set $\TT (N_{2,1})$ is reduced to a unique element, $\beta$, and $\MM(N_{2,1}) \simeq \Z \rtimes \Z/2\Z$, where the copy of $\Z$ is generated by $t_\beta$.
\item[(3)]
If $M$ is a connected surface, with negative Euler characteristic, different from $S_{0,3}$ and $N_{1,2}$, then $\TT (M)$ is non-empty. 
\end{itemize}
\end{lem}

The following lemma is well-known for orientable surfaces (see Fathi--Laudenbach--Poénaru \cite[Thm. III, Exp. 12]{FaLaPo1} and 
Ivanov \cite[Lem. 8.13]{Ivan1}).
It can be easily proved for a non-orientable surface $N$ from the fact the the lift of a pseudo-Anosov element of $\MM(N)$ to the mapping class group of the orientable double-covering of $N$ is a pseudo-Anosov element (see Wu \cite{Wu1}).

\begin{lem}\label{lem2_11}
Let $M$ be a connected surface (orientable or non-orientable), and let $f \in \MM(M)$ be a pseudo-Anosov element. 
Then the centralizer of $f$ in $\MM(M)$ is virtually cyclic.
\end{lem}

\begin{proof}[Proof of Proposition \ref{prop2_7}]
Let $H$ be an abelian subgroup of $\MM(N)$.
Set $\SS = \SS(H)$ and $\SS_\TT = \SS \cap \TT(N)$.
Denote by $\Lambda_\SS : \MM_\SS(N) \to \MM(N_\SS)$ the reduction homomorphism along $\SS$.
Note that $H \subset \MM_\SS(N)$ since $\SS$ is a reduction system for $H$.
By Proposition \ref{prop2_1} we have the following short exact sequence 
\[
1 \to H \cap Z_\SS \longrightarrow H \longrightarrow \Lambda_\SS(H) \to 1\,,
\]
hence $\rk (H) = \rk(H \cap Z_\SS) + \rk(\Lambda_\SS(H))$.
Moreover, again by Proposition \ref{prop2_1}, $\rk(H \cap Z_\SS) \le \rk(Z_\SS) = |\SS_\TT|$.

We  denote by $\Gamma (N_\SS)$ the set of connected components of $N_\SS$, by $\SSS(\Gamma(N_\SS))$ the group of permutations of $\Gamma(N_\SS)$, and by $\varphi: \MM(N_\SS) \to \SSS(\Gamma(N_\SS))$ the natural homomorphism.
Note that $\rk(\Lambda_\SS(H) \cap \Ker\,\varphi)=\rk (\Lambda_\SS(H))$ since $\Lambda_\SS(H) \cap \Ker\,\varphi$ is of finite index in $\Lambda_\SS(H)$.
Let $f \in \Lambda_\SS(H) \cap \Ker\,\varphi$ and let $N'$ be a connected component of $N_\SS$.
Let $f' \in H$ such that $f=\Lambda_\SS(f')$.
The set $\SS$ is an adequate reduction system for $f'$, hence the restriction of $f$ to $N'$ is either pseudo-Anosov or periodic.   
Moreover, by Lemma \ref{lem2_10}, $\MM(N')$ does not contain any pseudo-Anosov element if $N'$ is homeomorphic to $S_{0,3}$ or $N_{1,2}$, and, by Lemma \ref{lem2_11}, the centralizer of a pseudo-Anosov element of $\MM (N')$ is virtually cyclic. 
Let $N_1, \dots, N_\ell$ be the connected components of $N_\SS$ that are not homeomorphic to $S_{0,3}$ or $N_{1,2}$.
Then, by the above, $\rk(\Lambda_\SS(H)) = \rk(\Lambda_\SS(H) \cap \Ker\,\varphi) \le \ell$.

Let $i \in \{1, \dots, \ell\}$.
By Lemma \ref{lem2_10}, we have $\TT(N_i) \neq \emptyset$, hence we can choose a class $\beta_i \in \TT(N_i)$.
We set $\BB=\{\beta_1, \dots, \beta_\ell\}$ and $\AA=\SS_\TT \sqcup \BB$. 
Then $\AA$ is a simplex of $\TT(N)$, hence, by Corollary \ref{corl2_9}, we have $|\AA| \le \rk(N)$, and therefore  
\[
\rk(H) = \rk(H \cap Z_\SS) + \rk(\Lambda_\SS(H)) \le |\SS_\TT| + |\BB| = |\AA| \le \rk(N)\,.
\]

Assume that $\rk(H) = \rk(N)$.
Then $|\AA| = \rk(N)$ and $\SS_\TT = \SS(H) \cap \TT (N) \subset \AA$.
Moreover, $H \cap Z_\SS$ is of finite index in $Z_\SS$, hence there exists $k \ge 1$ such that $Z_\SS[k] \subset H \cap Z_\SS$.
Finally, for every $i \in \{1, \dots, \ell\}$, there exists $f \in \Lambda_\SS(H) \cap \Ker\, \varphi$ such that the restriction of $f$ to $N_i$ is pseudo-Anosov. 
By Lemma \ref{lem2_10}, this implies that $N_i$ cannot be homeomorphic to $N_{2,1}$.
\end{proof}


\section{Super-injective simplicial maps}\label{Sec3}

Let $\BB$ be a subset of $\CC(N)$.
A map $\lambda : \BB \to \BB$ is called a \emph{super-injective simplicial map} if the condition $i (\lambda (\alpha), \lambda (\beta)) = 0$ is equivalent to the condition $i (\alpha, \beta) = 0$ for all $\alpha, \beta \in \BB$.
It is shown in Irmak--Paris \cite[Lemma 2.2]{IrmPar1} that a super-injective simplicial map on $\TT(N)$ is always injective.
We will see that the same is true for the set $\BB = \TT_0(N)$ defined below (see Lemma \ref{lem3_5bis}).

Let $\alpha \in \TT(N)$.
We say that $\alpha$ \emph{bounds a Klein bottle} if  $N_\alpha$ has two connected components and one of them, $N'$, is a one-holed Klein bottle. 
Recall that $\TT (N')$ is a singleton $\{\alpha_0\}$ (see Lemma \ref{lem2_10}\,(2)).
The class $\alpha_0$ is then called the \emph{interior class} of $\alpha$.
The aim of this section is to prove the following theorem.
This together with Theorem 1.1 of Irmak--Paris \cite{IrmPar1} are the main ingredients for the proof of Theorem \ref{thm1_1}.

\begin{thm}\label{thm3_1}
Let $\GG$ be a finite index subgroup of $\MM(N)$ and let $\varphi: \GG \to \MM(N)$ be an injective homomorphism. 
There exists a super-injective simplicial map $\lambda : \TT(N) \to \TT(N)$ satisfying the following properties. 
\begin{itemize}
\item[(a)]
Let $\alpha \in \TT(N)$ which does not bound a Klein bottle. 
There exist $k,\ell \in \Z \setminus \{0\}$ such that $t_\alpha^k \in \GG$ and $\varphi(t_\alpha^k) = t_{\lambda(\alpha)}^\ell$.
\item[(b)]
Let $\alpha \in \TT(N)$ which bounds a Klein bottle and let $\alpha_0$ be its interior class.
The class $\lambda (\alpha)$ bounds a Klein bottle, $\lambda (\alpha_0)$ is the interior class of $\lambda (\alpha)$, and there exist $k, \ell\in \Z \setminus \{ 0 \}$ and $\ell_0 \in \Z$ such that $t_\alpha^k \in \GG$ and $\varphi( t_\alpha^k) = t_{\lambda (\alpha)}^\ell t_{\lambda(\alpha_0)}^{\ell_0}$.
\end{itemize}
\end{thm}

\begin{rem}
It will follow from Theorem \ref{thm1_1} that $\ell_0=0$ in Part (b) of the above theorem.
\end{rem}

We say that a class $\alpha \in \TT(N)$ is \emph{separating} (resp. \emph{non-separating}) if $N_\alpha$ has two connected components (resp. has one connected component). 
We denote by $\TT_1 (N)$ the set of separating classes $\alpha \in \TT (N)$ such that both components of $N_\alpha$ are non-orientable of odd genus.  
We denote by $\TT_2(N)$ the set of separating classes $\alpha \in \TT(N)$ that bound Klein bottles. 
And we denote by $\TT_0 (N)$ the complement of $\TT_1 (N) \cup \TT_2 (N)$ in $\TT (N)$.

The rest of the section is devoted to the proof of Theorem \ref{thm3_1}.
In Subsection \ref{subsec3_1} we prove the restriction of Theorem \ref{thm3_1} to $\TT_0 (N)$ (see Proposition \ref{prop3_4}).
In Subsection \ref{subsec3_2} we prove several combinatorial properties of super-injective simplicial maps $\TT_0 (N) \to \TT_0 (N)$. 
Theorem \ref{thm3_1} is proved in Subsection \ref{subsec3_3}. 


\subsection{Injective homomorphisms and super-injective simplicial maps of $\TT_0(N)$}\label{subsec3_1}

We define the \emph{rank} of an orientable surface $S_{\rho,n}$ to be $\rk(S_{\rho,n}) = 3\rho + n -3$.
It is well-known that $\rk(S_{\rho,n})$ is the maximal cardinality of a simplex of $\TT (S_{\rho,n})$ ($= \CC(S_{\rho,n})$).
By Corollary \ref{corl2_9} the same is true for non-orientable surfaces.

\begin{lem}\label{lem3_2}
Let $\alpha \in \TT(N)$.
\begin{itemize}
\item[(1)]
We have $\alpha \in \TT_1 (N)$ if and only if there is no simplex $\AA$ in $\TT (N)$ of cardinality $\rk(N)$ containing $\alpha$.
\item[(2)]
We have $\alpha \in \TT_2(N)$ if and only if there exists a simplex $\AA$ in $\TT(N)$ of cardinality $\rk(N)$ containing $\alpha$ and there exists a class $\beta \in \TT(N)$, different from $\alpha$, such that every simplex $\AA$ in $\TT (N)$ of cardinality $\rk (N)$ containing $\alpha$ also contains $\beta$.
\item[(3)]
We have $\alpha \in \TT_0 (N)$ if and only if there are two simplices $\AA$ and $\AA'$ in $\TT (N)$ of cardinality $\rk(N)$ such that $\AA \cap \AA' = \{ \alpha \}$.
\end{itemize}
\end{lem}

\begin{proof}
Suppose that $\alpha$ is non-separating.  
It is easily seen that $\rk(N) = \rk(N_\alpha) +1$.
Take two disjoint simplices $\AA_1$ and $\AA_1'$ in $\TT(N_\alpha)$ of cardinality $\rk(N_\alpha)$ and set $\AA= \AA_1 \cup \{\alpha\}$ and $\AA' = \AA_1' \cup \{\alpha\}$. 
Then $|\AA| = |\AA'| = \rk(N)$ and $\AA \cap \AA' = \{\alpha\}$.

Suppose that $\alpha$ is separating.  
Let $N_1$ and $N_2$ be the connected components of $N_\alpha$.
It is easily seen that $\rk(N) = \rk(N_1) + \rk(N_2) +2$ if $\alpha \in \TT_1(N)$ and $\rk(N) = \rk(N_1) + \rk(N_2) +1$ otherwise. 

Suppose that $\alpha \in \TT_1(N)$. 
Let $\AA$ be a simplex of $\TT(N)$ containing $\alpha$. 
Set $\AA_1 = \AA \cap \TT(N_1)$ and $\AA_2 = \AA \cap \TT(N_2)$.
We have $\AA = \AA_1 \sqcup \AA_2 \sqcup \{\alpha\}$, hence  
\[
|\AA| = |\AA_1| + |\AA_2| +1 \le \rk(N_1) + \rk(N_2) + 1 = \rk(N)-1 < \rk(N)\,.
\]

Suppose that $\alpha \in \TT_2 (N)$.
Then one of the connected components of $N_\alpha$, say $N_1$, is a one-holed Klein bottle. 
By Lemma \ref{lem2_10}\,(2), $\TT(N_1)$ contains a unique element, $\beta$. 
Let $\AA_2$ be a simplex of $\TT(N_2)$ of cardinality $\rk(N_2)$.
Set $\AA = \AA_2 \cup \{\alpha, \beta \}$. 
Then $\AA$ contains $\alpha$ and $|\AA| = \rk(N)$.
Let $\AA$ be a simplex of $\TT(N)$ of cardinality $\rk (N)$ containing $\alpha$.
Set $\AA_1 = \AA \cap \TT(N_1)$ and $\AA_2 = \AA \cap \TT(N_2)$.
Since $\AA = \AA_1 \sqcup \AA_2 \sqcup \{\alpha\}$ and $|\AA_2| \le \rk(N_2) = \rk(N)-2$, we have $\AA_1 \neq \emptyset$, hence $\AA_1 = \{ \beta\}$, and therefore $\beta \in \AA$.

Suppose that $\alpha \in \TT_0(N)$.
It is easily seen that there exist two disjoint simplices $\AA_1$ and $\AA_1'$ in $\TT(N_1)$ of cardinality $\rk(N_1)$.
Similarly, one can find two disjoint simplices $\AA_2$ and $\AA_2'$ in $\TT(N_2)$ of cardinality $\rk (N_2)$. 
Set $\AA = \AA_1 \cup \AA_2 \cup \{ \alpha \}$ and $\AA' = \AA_1' \cup \AA_2' \cup \{ \alpha \}$.
Then $\AA$ and $\AA'$ are two simplices of $\TT(N)$ of cardinality $\rk(N)$ such that $\AA \cap \AA' = \{ \alpha \}$. 
\end{proof}

For $\alpha \in \TT (N)$ we set $R(\alpha) = \{ f\in \MM(N) \mid \text{there exists } \ell \in \Z \text{ such that } f^\ell \in \langle t_\alpha \rangle \setminus \{1 \}\}$.
Note that, by Lemma \ref{lem2_4} and Lemma \ref{lem2_5}, we have $\SS(f)=\{\alpha\}$ for all $f \in R(\alpha)$.
We turn now to prove an algebraic characterization of the Dehn twists along the elements of $\TT_0 (N)$, up to roots and powers.  
This result and its proof are independent from the characterizations given in Atalan \cite{Atala1,Atala2} and Atalan--Szepietowski \cite{AtaSze1}, and they are interesting by themselves.

\begin{prop}\label{prop3_3}
\begin{itemize}
\item[(1)]
Let $f$ be an element of infinite order in $\MM(N)$.
If there exist two abelian subgroups $H$ and $H'$ in $\MM(N)$ of rank $\rk(N)$ such that $H \cap H' = \langle f \rangle$, then there exists $\alpha \in \TT_0 (N)$ such that $f \in R(\alpha)$.
\item[(2)]
Let $\GG$ be a finite index subgroup of $\MM(N)$, let $\alpha \in \TT_0 (N)$, and let $k \in \Z \setminus \{ 0\}$ such that $t_\alpha^k \in \GG$.
Then there exist two abelian subgroups $H$ and $H'$ in $\GG$ of rank $\rk(N)$ such that $H \cap H' = \langle t_\alpha^k \rangle$.
\end{itemize}
\end{prop}

\begin{proof}
We take an element $f \in \MM(N)$ of infinite order such that there exist two abelian subgroups $H$ and $H'$ of rank $\rk (N)$ satisfying $H \cap H' = \langle f \rangle$. 
First we prove that $\SS(H) \cap \TT (N) \neq \emptyset$.
Suppose that $\SS(H) \cap \TT (N) = \emptyset$.
This means that all the elements of $\SS(H)$ are one-sided. 
Let $\Lambda_{\SS(H)} : \MM_{\SS(H)}(N) \to \MM(N_{\SS(H)})$ be the reduction homomorphism along $\SS(H)$.
Then $N_{\SS(H)}$ is connected and, by Proposition \ref{prop2_1}, $\Lambda_{\SS(H)}$ is injective.
On the other hand the set $\SS(H)$ is an adequate reduction system for all $h \in H$, hence $\Lambda_{\SS(H)}(h)$ is either periodic or pseudo-Anosov for all $h \in H$.  
By Lemma \ref{lem2_11} it follows that $\rk (\Lambda_{\SS(H)}(H)) = \rk(H) \le 1$.
This contradicts the hypothesis $\rho\ge 5$ which implies $\rk(N) \ge 2$.

Now we show that $\SS(f) \cap \TT(N) \neq \emptyset$.
Suppose that $\SS(f) \cap \TT (N) = \emptyset$.
Then all the elements of $\SS(f)$ are one-sided, $N_{\SS(f)}$ is connected, and the reduction homomorphism $\Lambda_{\SS(f)} : \MM_{\SS(f)} (N) \to \MM(N_{\SS(f)})$ is injective. 
Since $\SS(f)$ is an adequate reduction system for $f$, the element $\Lambda_{\SS(f)}(f)$ is either periodic or pseudo-Anosov. 
We know by the above that $\SS(H) \cap \TT (N)$ is non-empty, hence we can choose a class $\alpha \in \SS(H) \cap \TT (N)$.
We have $\alpha \not\in \SS(f)$, since it is two-sided, hence $\alpha$ lies in $\TT (N_{\SS(f)})$ and it is a reduction class for $\Lambda_{\SS(f)} (f)$.
So, $\Lambda_{\SS(f)}(f)$ is periodic, hence $f$ is of finite order: contradiction.

Now we show that there exists $\alpha \in \TT(N)$ such that $f \in R (\alpha)$. 
By the above, we can choose a class $\alpha \in \SS(f) \cap \TT (N)$.
By Proposition \ref{prop2_7} there exists $k \ge 1$ such that $Z_{\SS(H)}[k] \subset H$, hence $t_\alpha^k \in H$.
Similarly there exists $k' \ge 1$ such that $t_\alpha^{k'} \in H'$.
So, $t_\alpha^{kk'} \in H \cap H'$.
Since $H \cap H' = \langle f \rangle$, it follows that there exists $\ell \in \Z\setminus \{0\}$ such that $f^\ell = t_\alpha^{kk'}$.

Since $\alpha \in \SS(H)$, by Proposition \ref{prop2_7} there exists a simplex $\AA \subset \TT(N)$ of cardinality $\rk (N)$ containing $\alpha$.
By Lemma \ref{lem3_2} it follows that $\alpha \not\in \TT_1(N)$.
Suppose that $\alpha \in \TT_2 (N)$.
Then $N_\alpha$ has two connected components, one of which, $N'$, is a one-holed Klein bottle. 
By Lemma \ref{lem2_10}, $\TT (N')$ is a singleton $\{ \beta \}$. 
If $\beta$ was not an element of $\SS (H)$, then $N'$ would be a connected component of $N_{\SS (H)}$ and this would contradict Proposition \ref{prop2_7}\,(3).
Hence $\beta \in \SS(H)$.
Moreover, again by Proposition \ref{prop2_7}, there exists $p \ge 1$ such that $t_\beta^p \in H$.
Similarly, there exists $p' \ge 1$ such that $t_\beta^{p'} \in H'$.
So, $t_\beta^{pp'} \in H \cap H'$.
On the other hand, we know that there exist $\ell \ge 1$ and $u \in \Z \setminus \{0\}$ such that $f^\ell = t_\alpha^u \in H \cap H'$.
Then the elements $t_\alpha^u, t_\beta^{pp'}$ generate a free abelian group of rank $2$ lying in $H \cap H' \simeq \Z$: contradiction.
So, $\alpha \in \TT_0 (N)$.

Now we consider a finite index subgroup $\GG$ of $\MM(N)$, a class $\alpha \in \TT_0 (N)$, and an integer $k \in \Z \setminus \{ 0 \}$ such that $t_\alpha^k \in \GG$. 
By Lemma \ref{lem3_2} there exist two simplices $\AA=\{\alpha_1, \dots, \alpha_r\}$ and $\AA' = \{\alpha_1', \dots, \alpha_r'\}$ in $\TT (N)$ of cardinality $r=\rk (N)$ such that $\AA \cap \AA' = \{\alpha\}$.
We assume that $\alpha = \alpha_1 = \alpha_1'$.
Since $\GG$ is of finite index in $\MM(N)$, there exists $\ell\ge 1$ such that $t_{\alpha_i}^\ell, t_{\alpha_i'}^\ell \in \GG$ for all $i \in \{2, \dots, r\}$.
Let $H$ (resp. $H'$) be the subgroup of $\GG$ generated by $t_\alpha^k, t_{\alpha_2}^\ell, \dots, t_{\alpha_r}^\ell$ (resp. $t_\alpha^k, t_{\alpha_2'}^\ell, \dots, t_{\alpha_r'}^\ell$).
Then $H$ and $H'$ are free abelian subgroups of $\GG$ of rank $r=\rk(N)$.
It remains to show that $H \cap H'=\langle t_\alpha^k \rangle$.

Let $g \in H \cap H'$.
Since $g \in H$, there exist $u_1,u_2, \dots, u_r \in \Z$ such that $g=t_\alpha^{ku_1} t_{\alpha_2}^{\ell u_2} \cdots t_{\alpha_r}^{\ell u_r}$.
Similarly, there exist $v_1,v_2, \dots, v_r \in \Z$ such that $g=t_\alpha^{kv_1} t_{\alpha_2'}^{\ell v_2} \cdots t_{\alpha_r'}^{\ell v_r}$.
By Lemma \ref{lem2_5} we have $\SS(g)= \{ \alpha_i \mid u_i \neq 0\} = \{\alpha_i' \mid v_i \neq 0\}$.
Since $\AA \cap \AA' = \{\alpha\}$, we have $u_i=v_i=0$ for all $i \in \{2, \dots, r\}$, hence $g=t_\alpha^{ku_1} \in \langle t_\alpha^k \rangle$.
\end{proof}

\begin{prop}\label{prop3_4}
Let $\GG$ be a finite index subgroup of $\MM(N)$ and let $\varphi: \GG \to \MM(N)$ be an injective homomorphism. 
Then there exits a super-injective simplicial map $\lambda : \TT_0(N) \to \TT_0(N)$ satisfying the following. 
Let $\alpha \in \TT_0(N)$.  
There exist $k,\ell \in \Z \setminus \{0\}$ such that $t_\alpha^k \in \GG$ and $\varphi(t_\alpha^k) = t_{\lambda(\alpha)}^\ell$.
\end{prop}

\begin{proof}
Let $\alpha \in \TT_0 (N)$.
Choose $u \in \Z\setminus \{0\}$ such that $t_\alpha^u \in \GG$.
By Proposition \ref{prop3_3} there exist two abelian subgroups $H,H'$ of $\GG$ of rank $\rk(N)$ such that $H \cap H' = \langle t_\alpha^u \rangle$.
The subgroups $\varphi(H)$ and $\varphi(H')$ are abelian groups of rank $\rk(N)$ and $\varphi(H) \cap \varphi(H') = \langle \varphi(t_\alpha^u) \rangle$, hence, by Proposition \ref{prop3_3}, there exists $\beta \in \TT_0 (N)$ such that $\varphi(t_\alpha^u) \in R(\beta)$.
There exist $v,\ell \in \Z \setminus \{0\}$ such that $\varphi(t_\alpha^{uv}) = \varphi(t_\alpha^u)^v = t_\beta^\ell$.
Then we set $k=uv$ and $\lambda(\alpha) = \beta$.

Let $\alpha, \alpha' \in \TT_0 (N)$ and let $k, k' \in \Z \setminus \{ 0 \}$.
By Stukow \cite{Stuko1} we have that, if $t_{\alpha}^{k} = t_{\alpha'}^{k'}$, then $\alpha = \alpha'$, and we have $t_{\alpha}^{k} t_{\alpha'}^{k'} = t_{\alpha'}^{k'} t_{\alpha}^{k}$ if and only if $i(\alpha, \alpha')=0$.
We assume that $k,k'$ are such that $\varphi( t_\alpha^k) = t_{\lambda (\alpha)}^\ell$ and $\varphi(t_{\alpha'}^{k'}) = t_{\lambda (\alpha')}^{\ell'}$ for some $\ell, \ell' \in \Z \setminus \{ 0 \}$.
Then 
\[
i (\lambda (\alpha), \lambda (\alpha')) = 0 \ \Leftrightarrow\
t_{\lambda (\alpha)}^\ell t_{\lambda (\alpha')}^{\ell'} = t_{\lambda (\alpha')}^{\ell'} t_{\lambda (\alpha)}^\ell \ \Leftrightarrow\
t_\alpha^k t_{\alpha'}^{k'} = t_{\alpha'}^{k'} t_\alpha^k\ \Leftrightarrow\
i(\alpha, \alpha')=0\,.
\]
So, $\lambda$ is a super-injective simplicial map.
\end{proof}


\subsection{Super-injective simplicial maps of $\TT_0(N)$}\label{subsec3_2}

In this subsection $\lambda : \TT_0 (N) \to \TT_0 (N)$ denotes a given super-injective simplicial map. 
Most of the results proved here are proved in Irmak--Paris \cite{IrmPar1} for super-injective simplicial maps of $\TT(N)$.
However, we use in Irmak--Paris \cite{IrmPar1} configurations of circles containing elements of $\TT_2(N)$ or $\TT_1 (N)$, and, on the other hand, here we use results such as Lemma \ref{lem3_6} that are false for super-injective simplicial maps of $\TT (N)$.
So, it is not easy to deduce the results of this subsection from Irmak--Paris \cite{IrmPar1} without losing the reader, hence we give independent proofs from Irmak--Paris \cite{IrmPar1} up to an exception (Lemma \ref{lem3_10}). 

The following lemma is easy to show and its proof is left to the reader.
It will be often used afterwards, hence it is important to keep it in mind.

\begin{lem}\label{lem3_5}
Let $\AA$ be a simplex of $\TT (N)$ of cardinality $\rk (N)$.
If $\rho$ is even, then all the connected components of $N_\AA$ are homeomorphic to $S_{0,3}$.
If $\rho$ is odd, then one of the connected components of $N_\AA$ is homeomorphic to $N_{1,2}$ and the other components are homeomorphic to $S_{0,3}$.
\end{lem}

\begin{lem}\label{lem3_5bis}
The map $\lambda : \TT_0 (N) \to \TT_0 (N)$ is injective.
\end{lem}

\begin{proof}
Let $\alpha, \beta$ be two distinct elements of $\TT_0 (N)$.
If $i(\alpha, \beta) \neq 0$, then $i (\lambda (\alpha), \lambda (\beta)) \neq 0$, hence $\lambda (\alpha) \neq \lambda (\beta)$.
Suppose that $i (\alpha, \beta) = 0$.
It is easily shown that there exists $\gamma \in \TT_0 (N)$ such that $i (\alpha, \gamma) \neq 0$ and $i (\beta, \gamma) =0$.
Then $i (\lambda (\alpha), \lambda (\gamma)) \neq 0$ and $i (\lambda (\beta), \lambda (\gamma)) =0$, hence $\lambda (\alpha) \neq \lambda (\beta)$.
\end{proof}

\begin{lem}\label{lem3_6}
Let $\AA=\{ \alpha_1, \dots, \alpha_r\}$ be a simplex of $\TT_0 (N)$ of cardinality $r=\rk(N)$.
There exists $\beta \in \TT_0 (N)$ such that $i(\alpha_1,\beta) \neq 0$ and $i(\alpha_i, \beta)=0$ for all $i \in \{2, \dots, r\}$.
\end{lem}

\begin{proof}
We choose pairwise disjoint representatives $a_i \in \alpha_i$, $i \in \{1, \dots, r\}$.
We denote by $N_\AA$ the natural compactification of $N \setminus ( \cup_{i=1}^r a_i)$ and by $\pi_\AA: N_\AA \to N$ the gluing map.  
We denote by $a_1'$ and $a_1''$ the two components of $\pi_\AA^{-1}(a_1)$, by $P'$ the connected component of $N_\AA$ containing $a_1'$, and by $P''$ the connected component containing $a_1''$.

Suppose that $P' = P''$.
Set $P=\pi_\AA(P')=\pi_\AA(P'')$.
Then $P$ is homeomorphic to $S_{1,1}$ or $N_{2,1}$ and there exists $i \in \{2, \dots, r\}$ such that $a_i = \partial P$ (say $a_2 = \partial P$).
Since $\alpha_2 \in \TT_0 (N)$, $P$ cannot be homeomorphic to $N_{2,1}$, hence $P$ is homeomorphic to $S_{1,1}$.
Let $\beta \in \TT (N)$ be the class represented by the circle $b$ drawn in Figure \ref{fig3_1}\,(i).
Then $\beta \in \TT_0 (N)$,  $i(\alpha_1, \beta)=1$, and $i(\alpha_i, \beta)=0$ for all $i \in \{2, \dots, r\}$.

\begin{figure}[ht!]
\begin{center}
\begin{tabular}{cc}
\parbox[c]{2.8cm}{\includegraphics[width=2.4cm]{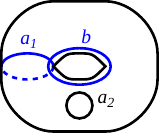}}
&
\parbox[c]{3.6cm}{\includegraphics[width=3.2cm]{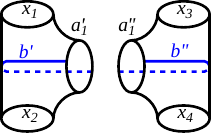}}
\\
(i) & (ii)
\end{tabular}

\caption{Circles in $N$ (Lemma \ref{lem3_6})}\label{fig3_1}
\end{center}
\end{figure}

Suppose that $P' \neq P''$.
By Lemma \ref{lem3_5}, either $P'$ and $P''$ are both homeomorphic to $S_{0,3}$, or one is homeomorphic to $S_{0,3}$ and the other is homeomorphic to $N_{1,2}$. 
The surfaces $P'$ and $P''$ are drawn in Figure \ref{fig3_1}\,(ii). 
In this figure each $x_i$ either is a boundary component of $N_\AA$,  or bounds a M\"obius band, and there is at most one $x_i$ bounding a M\"obius band. 
Consider the arcs $b'$ and $b''$ drawn in Figure \ref{fig3_1}\,(ii) and set $b=\pi_\AA (b' \cup b'')$ and $\beta = [b]$.
Then $i(\alpha_1, \beta)=2$ and $i(\alpha_i, \beta)=0$ for all $i \in \{2, \dots, r\}$.
Moreover, we have $\beta \not \in \TT_2 (N)$, since there is at most one $x_i$ bounding a M\"obius band, and we have $\beta \not \in \TT_1 (N)$, since $\{ \beta, \alpha_2, \dots, \alpha_r\}$ is a simplex in $\TT (N)$ of cardinality $r = \rk (N)$ (see Lemma \ref{lem3_2}), hence $\beta \in \TT_0 (N)$.
\end{proof}

Let $\AA=\{\alpha_1, \dots, \alpha_r\}$ be a simplex of $\TT_0 (N)$.
As before we denote by $\pi_\AA : N_\AA \to N$ the gluing map. 
We say that $\alpha_1$ and $\alpha_2$ are \emph{adjacent} with respect to $\AA$ if there exists a connected component $P$ of $N_\AA$ and two boundary components $a_1$ and $a_2$ of $P$ such that $\alpha_1 = [ \pi_\AA (a_1)]$ and $\alpha_2 = [ \pi_\AA (a_2)]$.

\begin{lem}\label{lem3_7}
Let $\AA= \{ \alpha_1, \dots, \alpha_r \}$ be a simplex of $\TT_0 (N)$ of cardinality $r=\rk(N)$.
If $\alpha_1$ and $\alpha_2$ are adjacent with respect to $\AA$ then there exists $\beta \in \TT_0 (N)$ such that $i(\alpha_1, \beta) \neq 0$, $i(\alpha_2, \beta) \neq 0$, and $i(\alpha_i, \beta)=0$ for all $i \in \{3, \dots, r\}$.
\end{lem}

\begin{proof}
We choose pairwise disjoint representatives $a_i \in \alpha_i$, $i \in \{1, \dots, r\}$.
We denote by $N_\AA$ the natural compactification of $N \setminus ( \cup_{i=1}^r a_i)$ and by $\pi_\AA: N_\AA \to N$ the gluing map.  
We denote by $a_i'$ and $a_i''$ the two components of $\pi_\AA^{-1}(a_i)$ for all $i$. 
We can assume that $a_1'$ and $a_2'$ are boundary components of the same connected component $P_0$ of $N_\AA$.
On the other hand, we denote by $P_1$ (resp. $P_2$) the connected component of $N_\AA$ containing $a_1''$ (resp. $a_2''$).

Suppose first that $P_0 = P_1$.
Set $P= \pi_\AA (P_0)$. 
Then $P$ is homeomorphic to $S_{1,1}$ or $N_{2,1}$, and $a_2$ is the boundary component of $P$.
But $\alpha_2 \in \TT_0 (N)$, hence $P$ cannot be homeomorphic to $N_{2,1}$, thus $P$ is homeomorphic to $S_{1,1}$.
There exists a subsurface $K$ of $N$ homeomorphic to $S_{1,2}$ endowed with the configuration of circles drawn in Figure \ref{fig3_2}\,(i).
Each boundary component of $K$ either is isotopic to an element of $\{a_3, \dots, a_r \}$, or is a boundary component of $N$, or bounds a M\"obius band. 
Moreover, we have $K \cap a_i = \emptyset$ for all $i \in \{3, \dots, r\}$.
Let $b$ be the circle drawn in Figure \ref{fig3_2}\,(i) and let $\beta = [b]$.
Then $\beta \in \TT_0 (N)$, $i(\alpha_1, \beta)=1$, $i (\alpha_2,\beta)=2$, and $i (\alpha_i,\beta)=0$ for all $i \in \{3, \dots, r\}$.

\begin{figure}[ht!]
\begin{center}
\begin{tabular}{ccc}
\parbox[c]{4cm}{\includegraphics[width=3.6cm]{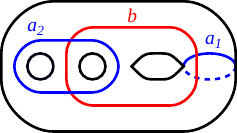}}
&
\parbox[c]{3cm}{\includegraphics[width=2.4cm]{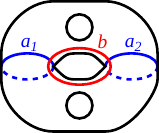}}
&
\parbox[c]{3.2cm}{\includegraphics[width=2.8cm]{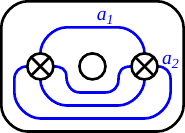}}
\\
(i) & (ii) & (iii)
\end{tabular}

\bigskip
\begin{tabular}{cc}
\parbox[c]{3.2cm}{\includegraphics[width=2.8cm]{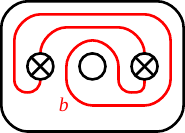}}
&
\parbox[c]{4.4cm}{\includegraphics[width=4cm]{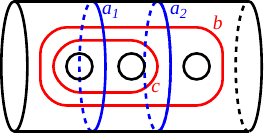}}
\\
(iv) & (v)
\end{tabular}

\caption{Circles in $N$ (Lemma \ref{lem3_7})}\label{fig3_2}
\end{center}
\end{figure}

Suppose that $P_1 = P_2$ and $P = \pi_\AA (P_0 \cup P_1)$ is orientable. 
There exists a subsurface $K$ of $N$ homeomorphic to $S_{1,2}$ endowed with the configuration of circles drawn in Figure \ref{fig3_2}\,(ii).
Each boundary component of $K$ either is isotopic to an element of $\{a_3, \dots, a_r \}$, or is a boundary component of $N$, or bounds a M\"obius band.
Moreover, we have $K \cap a_i = \emptyset$ for all $i \in \{3, \dots, r\}$.
Let $b$ be the circle drawn in Figure \ref{fig3_2}\,(ii) and let $\beta = [b]$.
Then $\beta \in \TT_0 (N)$, $i(\alpha_1, \beta)= i (\alpha_2,\beta)=1$, and $i (\alpha_i,\beta)=0$ for all $i \in \{3, \dots, r\}$.

Suppose that $P_1 = P_2$ and $P = \pi_\AA (P_0 \cup P_1)$ is non-orientable.
There exists a subsurface $K$ of $N$ homeomorphic to $N_{2,2}$ endowed with the configuration of circles drawn in Figure \ref{fig3_2}\,(iii).
In this figure and in all the others a disk with a cross inside represents a crosscap. 
This means that the disk containing the cross is removed and any two antipodal points in the resulting new boundary component are identified. 
In particular the circle represented by half of the boundary of the disk is one-sided. 
Each boundary component of $K$ either is isotopic to an element of $\{a_3, \dots, a_r \}$, or is a boundary component of $N$, or bounds a M\"obius band. 
Moreover, we have $K \cap a_i = \emptyset$ for all $i \in \{3, \dots, r\}$.
Let $b$ be the circle drawn in Figure \ref{fig3_2}\,(iv) and let $\beta = [b]$.
Then $\beta \in \TT_0 (N)$, $i(\alpha_1, \beta)=2$, $i (\alpha_2,\beta)=4$, and $i (\alpha_i,\beta)=0$ for all $i \in \{3, \dots, r\}$.

Suppose that $P_0 \neq P_1 \neq P_2 \neq P_0$.
There exists a subsurface $K$ of $N$ homeomorphic to $S_{0,5}$ endowed with the configuration of circles drawn in Figure \ref{fig3_2}\,(v).
Each boundary component of $K$ either is isotopic to an element of $\{a_3, \dots, a_r \}$, or is a boundary component of $N$, or bounds a M\"obius band. 
Moreover, by Lemma \ref{lem3_5}, at most one of the boundary components of $K$ bounds a M\"obius band. 
We also have $K \cap a_i = \emptyset$ for all $i \in \{3, \dots, r\}$.
Let $b$ and $c$ be the circles drawn in Figure \ref{fig3_2}\,(v).
Set $\beta = [b]$ and $\gamma = [c]$.
We have $i(\alpha_1, \beta)= i (\alpha_2,\beta)=2$ and $i (\alpha_i,\beta)=0$ for all $i \in \{3, \dots, r\}$.
We have $\beta \not\in \TT_1 (N)$, since $\{\beta, \gamma, \alpha_3, \dots, \alpha_r\}$ is a simplex of $\TT (N)$ of cardinality $\rk(N)$ containing $\beta$ (see Lemma \ref{lem3_2}), and we have $\beta \not \in \TT_2 (N)$, otherwise at least two boundary components of $K$ would bound M\"obius bands.  
So, $\beta \in \TT_0 (N)$.
\end{proof}

\begin{lem}\label{lem3_8}
Let $\AA$ be a simplex of $\TT_0 (N)$ of cardinality $\rk (N)$ and let $\alpha_1, \alpha_2 \in \AA$.
Then $\alpha_1$ is adjacent to $\alpha_2$ with respect to $\AA$ if and only if $\lambda (\alpha_1)$ is adjacent to $\lambda (\alpha_2)$ with respect to $\lambda (\AA)$.
\end{lem}

\begin{proof}
Set $\AA=\{\alpha_1, \alpha_2, \alpha_3, \dots, \alpha_r\}$.
Suppose that $\alpha_1$ is not adjacent to $\alpha_2$ with respect to $\AA$.
By Lemma \ref{lem3_6} there exist $\beta_1, \beta_2 \in \TT_0 (N)$ such that $i(\beta_1,\alpha_1) \neq 0$, $i(\beta_1, \alpha_i) = 0$ for all $i \in \{2, 3, \dots, r\}$, $i(\beta_2, \alpha_2) \neq 0$, and $i(\beta_2, \alpha_i) = 0$ for all $i \in \{1, 3, \dots, r\}$.
Since $\alpha_1$ is not adjacent to $\alpha_2$ with respect to $\AA$ we also have $i(\beta_1, \beta_2)=0$.
Since $\lambda$ is a super-injective simplicial map, we have $i(\lambda(\beta_1),\lambda(\alpha_1)) \neq 0$, $i(\lambda(\beta_1), \lambda(\alpha_i)) = 0$ for all $i \in \{2, 3, \dots, r\}$, $i(\lambda(\beta_2), \lambda(\alpha_2)) \neq 0$, $i(\lambda(\beta_2), \lambda(\alpha_i)) = 0$ for all $i \in \{1, 3, \dots, r\}$, and $i(\lambda(\beta_1), \lambda(\beta_2)) = 0$.
This is possible only if $\lambda(\alpha_1)$ is not adjacent to $\lambda(\alpha_2)$ with respect to $\lambda (\AA)$.

Suppose that $\alpha_1$ is adjacent to $\alpha_2$ with respect to $\AA$.
By Lemma \ref{lem3_7} there exists $\beta \in \TT_0 (N)$ such that $i(\beta,\alpha_1) \neq 0$, $i(\beta, \alpha_2) \neq 0$, and $i(\beta, \alpha_i) = 0$ for all $i \in \{3, \dots, r\}$.
Since $\lambda$ is a super-injective simplicial map, we have $i(\lambda(\beta),\lambda(\alpha_1)) \neq 0$, $i(\lambda(\beta), \lambda(\alpha_2)) \neq 0$, and $i(\lambda(\beta), \lambda(\alpha_i)) = 0$ for all $i \in \{3, \dots, r\}$.
This is possible only if $\lambda (\alpha_1)$ is adjacent to $\lambda (\alpha_2)$ with respect to $\lambda(\AA)$.
\end{proof}

Let $\PP=\{\alpha_1, \alpha_2, \alpha_3\}$ be a simplex of $\TT_0 (N)$ of cardinality $3$.
We choose pairwise disjoint representatives $a_1, a_2,a_3$ of $\alpha_1, \alpha_2, \alpha_3$, respectively.  
We say that $\PP$ is a \emph{3-simpants} if there exists a subsurface $P$ of $N$ homeomorphic to $S_{0,3}$ whose boundary components are $a_1, a_2, a_3$.
Let $\PP=\{\alpha_1, \alpha_2\}$ be a simplex of $\TT_0 (N)$ of cardinality $2$.
We choose disjoint representatives $a_1, a_2$ of $\alpha_1, \alpha_2$, respectively. 
We say that $\PP$ is a \emph{2-simpants} if there exists a subsurface $P$ of $N$ homeomorphic to $S_{0,3}$ whose boundary components are $a_1,a_2$ and a boundary component of $N$.
We say that $\PP$ is a \emph{simpskirt} if there exists a subsurface $P$ of $N$ homeomorphic to $N_{1,2}$ whose boundary components are $a_1,a_2$.
Let $\beta \in \TT_0(N)$.
We say that $\beta$ \emph{bounds a torus} if $N_\beta$ has two connected components, one of which, $N'$, is a one-holed torus.  
In this case the elements of $\TT(N')$ are called \emph{interior classes} of $\beta$.

\begin{lem}\label{lem3_9}
\begin{itemize}
\item[(1)]
If $\PP$ is a 3-simpants and $N_\PP$ has two connected components, then $\lambda(\PP)$ is a 3-simpants. 
\item[(2)]
If $\alpha \in \TT_0 (N)$ is non-separating and $N_\alpha$ is non-orientable, then $\lambda (\alpha)$ is non-separating.  
\item[(3)]
If $\beta$ bounds a torus and $\alpha$ is an interior class of $\beta$, then $\lambda(\beta)$ bounds a torus and $\lambda(\alpha)$ is an interior class of $\lambda(\beta)$.
\end{itemize}
\end{lem}

\begin{proof}
We take a subsurface $K$ of $N$ homeomorphic to $S_{1,2}$ endowed with the configuration of circles drawn in Figure \ref{fig3_3}\,(i).
We assume that $a_3,a_4$ are essential circles and $[a_3], [a_4] \in \TT_0 (N)$.
We set $\alpha_i = [a_i]$ for all $i \in \{1,2,3,4 \}$ and $\beta_2 = [b_2]$.
The following claims are easy to prove. 

{\it Claim 1.}
Suppose that $\PP = \{\gamma_1, \gamma_2, \gamma_3\}$ is a 3-simpants and that $N_\PP$ has two connected components. 
Then $\gamma_1, \gamma_2, \gamma_3$ are non-separating and at most one of the surfaces $N_{\gamma_i}$ is orientable.
Moreover, up to renumbering $\gamma_1, \gamma_2, \gamma_3$, we can choose $K$ so that $\gamma_i=\alpha_i$ for all $i \in \{1,2,3 \}$.
In fact, the only constrain we have is that, if $N_{\gamma_i}$ is orientable, then $\gamma_i= \gamma_3= \alpha_3$.

{\it Claim 2.}
If $\alpha$ is a non-separating class such that $N_\alpha$ is non-orientable, then we can choose $K$ so that $\alpha=\alpha_1$.

{\it Claim 3.}
If $\beta$ bounds a torus and $\alpha$ is an interior class of $\beta$, then we can choose $K$ so that $\beta= \beta_2$ and $\alpha=\alpha_1$.

\begin{figure}[ht!]
\begin{center}
\begin{tabular}{cccc}
\parbox[c]{2.8cm}{\includegraphics[width=2.4cm]{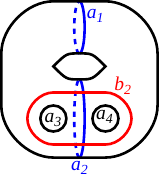}}
&
\parbox[c]{3.6cm}{\includegraphics[width=3.2cm]{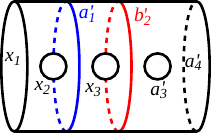}}
&
\parbox[c]{2.8cm}{\includegraphics[width=2.4cm]{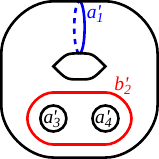}}
&
\parbox[c]{2.8cm}{\includegraphics[width=2.4cm]{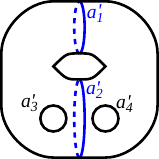}}
\\
(i) & (ii) & (iii) & (iv)
\end{tabular}

\caption{Circles in $N$ (Lemma \ref{lem3_9})}\label{fig3_3}
\end{center}
\end{figure}

We complete $\{ \alpha_1, \alpha_2, \alpha_3, \alpha_4 \}$ in a simplex $\AA=\{ \alpha_1 , \dots, \alpha_r\}$ of $\TT_0 (N)$ of cardinality $r = \rk(N)$.
Let $\BB=\{ \alpha_1, \beta_2, \alpha_3, \alpha_4, \dots, \alpha_r\}$.
Then $\BB$ is also a simplex of $\TT_0 (N)$ of cardinality $r=\rk(N)$.
We choose minimally intersecting representatives $a_i' \in \lambda (\alpha_i)$ for all $i \in \{1, \dots, r\}$ and $b_2' \in \lambda (\beta_2)$.  
By Lemma \ref{lem3_8} the elements of $\lambda(\BB)$ that are adjacent to $\lambda(\beta_2)$ with respect to $\lambda (\BB)$ are exactly $\lambda(\alpha_1), \lambda (\alpha_3), \lambda (\alpha_4)$, the class $\lambda (\alpha_1)$ is adjacent only to $\lambda (\beta_2)$ with respect to $\lambda (\BB)$, and the class $\lambda (\alpha_3)$ is adjacent to $\lambda (\alpha_4)$ with respect to $\lambda (\BB)$.
It follows that we have one of the following alternatives. 
\begin{itemize}
\item
There exists a subsurface $K'$ of $N$ homeomorphic to $S_{0,5}$ endowed with the configuration of circles drawn in Figure \ref{fig3_3}\,(ii), where each circle $x_i$ either is a boundary component of $N$, or bounds a M\"obius band.
\item
There exists a subsurface $K'$ of $N$ homeomorphic to $S_{1,2}$ endowed with the configuration of circles drawn in Figure \ref{fig3_3}\,(iii).
\end{itemize}

Since $i(\lambda(\alpha_2), \lambda(\alpha_i))=0$ for all $i \in \{1,3, \dots, r\}$ and $i(\lambda(\alpha_2), \lambda(\beta_2)) \neq 0$, the circle $a_2'$ must lie in $K'$.
By Lemma \ref{lem3_8}, $\lambda(\alpha_1)$ is adjacent to $\lambda(\alpha_3)$ and $\lambda(\alpha_4)$ with respect to $\lambda(\AA)$, hence $a_1', a_3', a_4'$ are included in the same connected component of $K' \setminus a_2'$.
Such a circle $a_2'$ does not exist in the configuration of Figure \ref{fig3_3}\,(ii), hence $K'$ is homeomorphic to $S_{1,2}$ and $a_1', b_2', a_3', a_4'$ are as shown in Figure \ref{fig3_3}\,(iii). 
This shows Part (2) and Part (3) of the lemma. 
This also shows that, up to homeomorphism, $K'$ is endowed with the configuration of circles drawn in Figure \ref{fig3_3}\,(iv), which shows Part (1). 
\end{proof}

The next lemma can be proved in the same way as Irmak--Paris \cite[Lemma 2.7]{IrmPar1} (see also Ivanov \cite[Lemma 1]{Ivano2}), so we do not give any poof. 

\begin{lem}\label{lem3_10}
Let $\alpha_1, \alpha_2 \in \TT_0 (N)$.
If $i (\alpha_1, \alpha_2) = 1$, then $i (\lambda (\alpha_1), \lambda (\alpha_2)) = 1$.
\end{lem}

\begin{lem}\label{lem3_11}
Suppose that $\rho$ is even. 
\begin{itemize}
\item[(1)]
Let $\PP$ be a $2$-simpants such that $N_\PP$ has two connected components and the connected component of $N_\PP$ non-homeomorphic to $S_{0,3}$ is non-orientable. 
Then $\lambda (\PP)$ is a 2-simpants. 
\item[(2)]
Let $\PP$ be a simpskirt such that $N_\PP$ has two connected components. 
Then $\lambda (\PP)$ is a simpskirt.
\end{itemize}
\end{lem}

\begin{proof}
Let $\PP=\{\alpha_1, \alpha_2\}$ be a 2-simpants such that $N_\PP$ has two connected components and the connected component of $N_\PP$ non-homeomorphic to $S_{0,3}$ is non-orientable. 
We take a subsurface $K$ of $N$ homeomorphic to $S_{1,3}$ endowed with the configuration of circles drawn in Figure \ref{fig3_4}\,(i). 
In this configuration $a_1,a_2, a_3,a_4,a_5$ are essential circles that represent elements of $\TT_0 (N)$ and $x$ is a boundary component of $N$.
We can and do choose the configuration so that $\alpha_1 = [a_1]$ and $\alpha_2 = [a_2]$.
We set $\alpha_i = [a_i]$ for $i \in \{3,4,5\}$ and we complete $\{\alpha_1, \dots, \alpha_5\}$ in a simplex $\AA = \{\alpha_1, \dots, \alpha_r \}$ of $\TT_0 (N)$ of cardinality $r = \rk (N)$.
We set $\beta_i = \lambda (\alpha_i)$ and we choose $b_i \in \beta_i$ such that $b_1, \dots, b_r$ are pairwise disjoint.
We denote by $N_{\lambda (\AA)}$ the natural compactification of $N \setminus (\cup_{i=1}^r b_i)$ and by $\pi_{\lambda (\AA)} : N_{\lambda (\AA)} \to N$ the gluing map. 
We denote by $b_i'$ and $b_i''$ the two components of $\pi_{\lambda (\AA)}^{-1} (b_i)$.
Recall that, by Lemma \ref{lem3_5}, each connected component of $N_{\lambda (\AA)}$ is a pair of pants.  
By Lemma \ref{lem3_8}, $\beta_3$ is adjacent to $\beta_1, \beta_2, \beta_4, \beta_5$ with respect to $\lambda(\AA)$ and $\beta_1$ is adjacent to $\beta_2$ but not to $\beta_3$ and $\beta_4$ with respect to $\lambda (\AA)$.
It follows that, up to replacing $b_i''$ by $b_i'$, there exists a connected component $P_1$ of $N_{\lambda (\AA)}$ whose boundary components are $b_1', b_2', b_3'$ and there exists another connected component $P_2$ of $N_{\lambda (\AA)}$ whose boundary components are $b_3'', b_4', b_5'$.
Let $P_3$ (resp. $P_4$) be the connected component of $N_{\lambda (\AA)}$ containing $b_1''$ (resp. $b_2''$) in its boundary.
By Lemma \ref{lem3_9}, $\beta_1$ is non-separating and by Lemma \ref{lem3_8} the elements of $\lambda (\AA)$ adjacent to $\beta_1$ with respect to $\lambda (\AA)$ are precisely $\beta_2, \beta_3$.
This implies that $P_3=P_4$ and the third boundary component of $P_3$ is sent to a boundary component of $N$ under $\pi_{\lambda (\AA)}$. 
So, $\lambda (\PP)$ is a 2-simpants. 

\begin{figure}[ht!]
\begin{center}
\begin{tabular}{cc}
\parbox[c]{2.8cm}{\includegraphics[width=2.4cm]{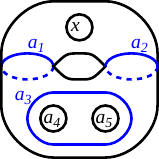}}
&
\parbox[c]{6.6cm}{\includegraphics[width=6.2cm]{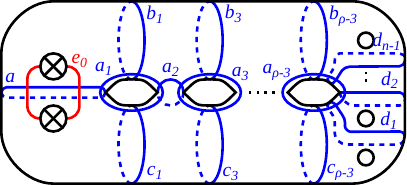}}
\\
(i) & (ii)
\end{tabular}

\bigskip
\begin{tabular}{c}
\parbox[c]{7.2cm}{\includegraphics[width=6.8cm]{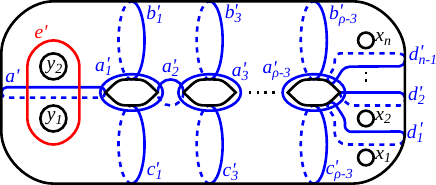}}
\\
(iii)
\end{tabular}

\caption{Circles in $N$ (Lemma \ref{lem3_11})}\label{fig3_4}
\end{center}
\end{figure}

Now we take a simpskirt $\PP$ such that $N_\PP$ has two connected components and we show that $\lambda (\PP)$ is a simpskirt. 
We assume that $n \ge 1$.
The case $n=0$ can be proved in the same way. 
Consider the configuration of circles in $N$ drawn in Figure \ref{fig3_4}\,(ii).
We can and do assume that $\PP = \{ [a], [b_1] \}$.
We choose $a' \in \lambda ([a]), a_1' \in \lambda ([a_1]), \dots, a_{\rho-3}' \in \lambda ([a_{\rho-3}]), b_1' \in \lambda ([b_1]), \dots, b_{\rho-3}' \in \lambda ([b_{\rho-3}]), c_1' \in \lambda ([c_1]), \dots, c_{\rho-3}' \in \lambda ([c_{\rho-3}]), d_1' \in \lambda ([d_1]), \dots, d_{n-1}' \in \lambda ([d_{n-1}])$ that minimally intersect.
By Lemma \ref{lem3_9}, Lemma \ref{lem3_10}, and the above, there exists an orientable subsurface $M$ of $N$ of genus $\frac{\rho-2}{2}$ endowed with the configuration of circles drawn in Figure \ref{fig3_4}\,(iii), where $x_1, \dots, x_n$ are the boundary components of $N$.
For topological reasons, the circle $e'$ bounds a one-holed Klein bottle $K$ in $N$.
This implies that either $y_1$ is isotopic to $y_2$, or $y_1,y_2$ both bound M\"obius bands. 
By Lemma \ref{lem2_10}\,(2), $\TT (K)$ is a singleton  $\{\epsilon_0'\}$.
Since $i(\lambda ([e_0]), [z'])=0$ for all $z' \in \{a_1', \dots, a_{\rho-3}', b_1', \dots, b_{\rho-3}', c_1', \dots, c_{\rho-3}', d_1', \dots, d_{n-1}' \}$ and $i( \lambda ([e_0]), [a']) \neq 0$, we have $\lambda ([e_0]) \in \TT (K)$, that is, $\lambda ([e_0]) = \epsilon_0'$ (we cannot have $\lambda ( [e_0]) = [e']$ since $[e'] \not\in \TT_0 (N)$).
If $y_1$ was isotopic to $y_2$, then we would have $[y_1] = \epsilon_0'$, hence $0 = i([y_1], [a']) = i(\lambda ([e_0]), [a']) \neq 0$: contradiction.
So, $y_1,y_2$ both bound M\"obius bands and therefore $\lambda (\PP) = \{ [a'], [b_1'] \}$ is a simpskirt.
\end{proof}

\begin{lem}\label{lem3_12}
Assume that $\rho$ is odd.
\begin{itemize}
\item[(1)]
Let $\PP$ be a 2-simpants such that $N_\PP$ has two connected components.  
Then $\lambda (\PP)$ is a 2-simpants. 
\item[(2)]
Let $\PP$ be a simpskirt such that $N_\PP$ has two connected components.
Then $\lambda (\PP)$ is a simpskirt.
\end{itemize}
\end{lem}

\begin{proof}
Let $\PP = \{ \alpha_1, \alpha_2 \}$ be a simplex of $\TT_0 (N)$ of cardinality $2$.
We suppose that $\PP$ is either a 2-simpants or a simpskirt, and that $N_\PP$ has two connected components.
We take a subsurface $K$ of $N$ homeomorphic to $S_{1,3}$ endowed with the configuration of circles drawn in Figure \ref{fig3_5}\,(i). 
In this configuration $a_1,a_2, a_3,a_4,a_5$ are essential circles that represent elements of $\TT_0 (N)$, the circle $x$ is a boundary component of $N$ if $\PP$ is a 2-simpants, and $x$ bounds a M\"obius band if $\PP$ is a simpskirt.
We can and do choose the configuration so that $\alpha_1 = [a_1]$ and $\alpha_2 = [a_2]$.
We set $\alpha_i = [a_i]$ for $i \in \{3,4,5\}$ and we complete $\{ \alpha_1, \dots, \alpha_5\}$ in a simplex $\AA = \{ \alpha_1, \dots, \alpha_r \}$ of $\TT_0 (N)$ of cardinality $r = \rk (N)$.
We set $\beta_i = \lambda (\alpha_i)$ and we choose $b_i \in \beta_i$ such that $b_1, \dots, b_r$ are pairwise disjoint. 
We denote by $N_{\lambda (\AA)}$ the natural compactification of $N \setminus (\cup_{i=1}^r b_i)$ and by $\pi_{\lambda (\AA)} : N_{\lambda (\AA)} \to N$ the gluing map.
We denote by $b_i'$ and $b_i''$ the two boundary components of $N_{\lambda (\AA)}$ that are sent to $b_i$ under $\pi_{\lambda (\AA)}$.
By Lemma \ref{lem3_8}, $\beta_3$ is adjacent to $\beta_1, \beta_2, \beta_4, \beta_5$ with respect to $\lambda(\AA)$ and $\beta_1$ is adjacent to $\beta_2$ but not to $\beta_3$ and $\beta_4$ with respect to $\lambda (\AA)$.
It follows that, up to replacing $b_i''$ by $b_i'$, there exists a connected component $P_1$ (resp. $P_2$) of $N_{\lambda (\AA)}$ which is a pair of pants and whose boundary components are $b_1', b_2', b_3'$ (resp. $b_3'', b_4', b_5'$).
Let $P_3$ (resp. $P_4$) be the connected component of $N_{\lambda (\AA)}$ containing $b_1''$ (resp. $b_2''$) in its boundary. 
By Lemma \ref{lem3_9}, $\beta_1$ is non-separating and by Lemma \ref{lem3_8} the only elements of $\lambda (\AA)$ adjacent to $\beta_1$ with respect to $\lambda (\AA)$ are $\beta_2, \beta_3$.
This implies that $P_3=P_4$ and either $P_3$ is a pair of pants and its third boundary component is sent to a boundary component of $N$ under $\pi_{\lambda (\AA)}$, or $P_3$ is homeomorphic to $N_{1,2}$.
We conclude that either $\lambda (\PP)$ is a 2-simpants, or $\lambda (\PP)$ is a simpskirt. 

\begin{figure}[ht!]
\begin{center}
\begin{tabular}{ccc}
\parbox[c]{2.8cm}{\includegraphics[width=2.4cm]{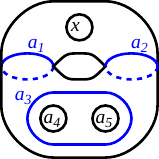}}
&
\parbox[c]{3.4cm}{\includegraphics[width=3cm]{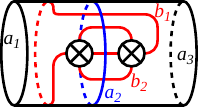}}
&
\parbox[c]{4.6cm}{\includegraphics[width=4.2cm]{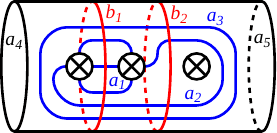}}
\\
(i) & (ii) & (iii)
\end{tabular}

\caption{Circles in $N$ (Lemma \ref{lem3_12})}\label{fig3_5}
\end{center}
\end{figure}

Suppose that $\PP$ is a simpskirt and that the connected component of $N_\PP$ non-homeomor\-phic to $N_{1,2}$ is non-orientable.  
Let $K$ be a subsurface of $N$ homeomorphic to $N_{2,2}$ endowed with the configuration of circles drawn in Figure \ref{fig3_5}\,(ii).
We suppose that the complement of $K$ in $N$ is connected. 
In particular, $a_1,a_2,a_3,b_1,b_2$ are non-separating, hence they represent elements of $\TT_0 (N)$. 
We can and do choose $K$ such that $\alpha_1 = [a_1]$ and $\alpha_2 = [a_2]$.
We set $\alpha_3 = [a_3]$, $\beta_1 = [b_1]$, and $\beta_2 = [b_2]$.
We choose minimally intersecting representatives $a_i' \in \lambda (\alpha_i)$ for $i \in \{1,2,3\}$ and $b_j' \in \lambda (\beta_j)$ for $j \in \{1,2\}$. 
By Lemma \ref{lem3_9}, $\{ \lambda(\alpha_1), \lambda (\beta_1), \lambda (\beta_2) \}$ and $\{ \lambda(\alpha_3), \lambda (\beta_1), \lambda (\beta_2) \}$ are 3-simpants.  
It follows that there exists a subsurface $K_1'$ of $N$ homeomorphic to $S_{1,2}$ or $N_{2,2}$, whose boundary components are $a_1'$ and $a_3'$, and which contains $b_1'$ and $b_2'$ in its interior.
On the other hand, by the above, each of $\{ \lambda (\alpha_1), \lambda (\alpha_2)\}$ and $\{ \lambda (\alpha_2), \lambda (\alpha_3)\}$ is either a 2-simpants or a simpskirt.  
It follows that there exists a subsurface $K_2'$ of $N$ homeomorphic either to $S_{0,4}$, or to $N_{1,3}$, or to $N_{2,2}$, whose boundary components are $a_1'$, $a_3'$ and possibly boundary components of $N$, and which contains $a_2'$ in its interior. 
Since $i(\lambda(\alpha_2), \lambda(\beta_2)) \neq 0$, we have $a_2' \cap b_2' \neq \emptyset$, hence $K_1'=K_2'$.
We conclude that $K_1'=K_2'$ is homeomorphic to $N_{2,2}$ and $\lambda (\PP)=\{\lambda (\alpha_1), \lambda(\alpha_2) \}$ is a simpskirt. 

Suppose that $\PP = \{\alpha_1, \alpha_2 \}$ is a simpskirt and the connected component of $N_\PP$ non-homeomorphic to $N_{1,2}$ is orientable.  
Let $K$ be a subsurface of $N$ homeomorphic to $N_{3,2}$ endowed with the configuration of circles drawn in Figure \ref{fig3_5}\,(iii). 
We assume that the complement of $K$ in $N$ is connected.  
In particular, $a_1,a_2,a_4,a_5,b_1,b_2$ are non-separating, and therefore they represent elements of $\TT_0 (N)$.
We choose $K$ so that $N_{[a_3]}$ has two connected components, a non-orientable one of genus $3$ and an orientable one. 
In particular, we also have $[a_3] \in \TT_0 (N)$. 
We can and do choose $K$ such that $\alpha_1 = [a_1]$ and $\alpha_2 = [a_2]$.
We set $\alpha_i = [a_i]$ for $i \in \{ 3,4,5 \}$, $\beta_j = [b_j]$ for $j \in \{ 1,2 \}$, and we complete $\{ \alpha_1, \dots, \alpha_5 \}$ in a simplex $\AA = \{ \alpha_1 , \dots, \alpha_r \}$ of $\TT_0 (N)$ of cardinality $r = \rk (N)$.
We choose minimally intersecting representatives $a_i' \in \lambda (\alpha_i)$ for $i \in \{1, \dots, r\}$ and $b_j' \in \lambda (\beta_j)$ for $j \in \{1,2\}$.  
By the above, $\{ \lambda (\alpha_4), \lambda (\beta_1) \}$, $\{ \lambda (\beta_1), \lambda (\beta_2) \}$, and $\{ \lambda (\beta_2), \lambda (\alpha_5) \}$ are simpskirts.  
It follows that there exists a subsurface $K_1'$ of $N$ homeomorphic to $N_{3,2}$ whose boundary components are $a_4',a_5'$ and which contains $b_1',b_2'$ in its interior.  
By Lemma \ref{lem3_8}, $\lambda (\alpha_3)$ is adjacent to $\lambda (\alpha_1), \lambda (\alpha_2), \lambda (\alpha_4), \lambda (\alpha_5)$ with respect to $\lambda (\AA)$, hence there exists a subsurface $R'$ of $N$ homeomorphic to $S_{0,4}$ whose boundary components are $a_1', a_2', a_4', a_5'$ and which contains $a_3'$ in its interior.
Recall that we know that $\{\lambda (\alpha_1), \lambda (\alpha_2) \}$ is either a 2-simpants or a simpskirt.  
If $\{ \lambda (\alpha_1), \lambda (\alpha_2) \}$ was a 2-simpants, then there would exist a subsurface $K_2'$ of $N$ homeomorhic to $S_{1,3}$ or $N_{2,3}$, containing $R'$, and whose boundary components are $a_4', a_5'$ and a boundary component of $N$.
Since $i (\alpha_3, \beta_1) \neq 0$, we would have $a_3' \cap b_1' \neq \emptyset$, hence $K_1' = K_2'$: contradiction.
So, $\{ \lambda (\alpha_1), \lambda (\alpha_2) \}$ is a simpskirt.

Suppose now that $\PP=\{\alpha_1, \alpha_2\}$ is a 2-simpants. 
We know that $\lambda (\PP)$ is a simpskirt or a 2-simpants.  
We complete $\PP=\{ \alpha_1, \alpha_2\}$ in a simplex $\AA = \{ \alpha_1, \dots, \alpha_r\}$ of $\TT_0 (N)$ of cardinality $r = \rk(N)$.
We can and do choose $\AA$ so that $\PP' = \{ \alpha_2, \alpha_3\}$ is a simpskirt and $N_{\PP'}$ has two connected components.  
By Lemma \ref{lem3_5} there is at most one pair $\{ \lambda (\alpha_i), \lambda (\alpha_j) \}$ in $\lambda (\AA)$ which is a simpskirt.  
By the above the pair $\{ \lambda (\alpha_2), \lambda (\alpha_3)\}$ is a simpskirt, hence $\{ \lambda (\alpha_1), \lambda (\alpha_2) \}$ is a 2-simpants. 
\end{proof}


\subsection{Proof of Theorem \ref{thm3_1}}\label{subsec3_3}

In this subsection $\GG$ denotes a finite index subgroup of $\MM (N)$ and $\varphi : \GG \to \MM (N)$ an injective homomorphism.
Let $\mu : \TT_0(N) \to \TT_0 (N)$ be the super-injective simplicial map induced by $\varphi$ (see Proposition \ref{prop3_4}).

The following lemma is well-known for orientable surfaces (see Castel \cite[Proposition 2.1.3]{Caste1} for example). 
It can be proved for non-orientable surfaces in the same way as for orientable surfaces.

\begin{lem}\label{lem3_13}
Let $\{x_1, \dots, x_\ell\}$, $\{y_1, \dots, y_\ell\}$ be two collections of essential circles such that:
\begin{itemize}
\item[(a)]
$x_1, \dots, x_\ell$ (resp. $y_1, \dots, y_\ell$) are pairwise non-isotopic and they minimally intersect; 
\item[(b)]
$x_i$ is isotopic to $y_i$ for all $i \in \{1, \dots, \ell\}$;
\item[(c)]
there exist no three indices $i,j,k \in \{1, \dots, \ell\}$ such that $i([x_i],[x_j]) \neq 0$, $i([x_j], [x_k]) \neq 0$, and $i ([x_k], [x_i]) \neq 0$.
\end{itemize}
Then there exists a homeomorphism $H: N \to N$ isotopic to the identity such that $H(x_i) = y_i$ for all $i \in \{1, \dots, \ell\}$.
\end{lem}

\begin{lem}\label{lem3_14}
Let $\gamma \in \TT_1 (N)$.
Suppose that one of the connected components of $N_\gamma$ is (non-orientable) of genus $1$.
Then there exist $\gamma' \in \TT_1 (N)$ and $z,z' \in \Z \setminus \{0\}$ such that $t_\gamma^z \in \GG$ and $\varphi(t_\gamma^z) = t_{\gamma'}^{z'}$.
\end{lem}

\begin{proof}
Consider the configuration of circles in $N$ drawn in Figure \ref{fig3_6}.
We set $\alpha_i = [a_i]$ for $i \in \{1, \dots, \rho-2\}$, $\beta = [b]$, $\gamma_j = [c_j]$ for $j \in \{1, \dots, n\}$, $\delta_j = [d_j]$ for $j \in \{1, \dots, n\}$, and $\epsilon_j = [e_j]$ for $j \in \{2, \dots, n\}$.
We can and do choose this configuration so that $\gamma=\gamma_j$ for some $j \in \{1, \dots, n\}$.
By the results of Subsection \ref{subsec3_2}, there exists a homeomorphism $H : N \to N$ such that $[H(u)] = \mu ([u])$ for all $u \in \{ a_1, \dots, a_{\rho-2}, b, d_1, \dots, d_n \}$.
We show by induction on $j$ that $[H (e_j)] = \mu (\epsilon_j)$ for all $j \in \{2, \dots, n\}$.
The class $\mu (\epsilon_2)$ is the unique element of $\TT_0 (N)$ that satisfies $i(\mu (\epsilon_2), \eta)=0$ for all $\eta \in \{ \mu (\alpha_1), \dots, \mu (\alpha_{\rho-2}), \mu (\beta), \mu (\delta_1), \mu(\delta_3), \dots, \mu(\delta_n) \}$.
Since $[H(e_2)]$ also satisfies this property, we have $\mu (\epsilon_2) = [H(e_2)]$.
We assume that $3 \le j \le n$ and $\mu (\epsilon_k) = [H (e_k)]$ for all $k \in \{2, \dots, j-1\}$.
The class $\mu (\epsilon_j)$ is the unique element of $\TT_0 (N)$ that satisfies $i (\mu (\epsilon_j), \eta) = 0$ and $\mu (\epsilon_j) \neq \eta$ for all $\eta \in \{\mu (\alpha_1), \dots, \mu (\alpha_{\rho-2}), \mu (\beta), \mu(\delta_1), \mu (\delta_{j+1}), \dots, \mu (\delta_n), \mu(\epsilon_2), \dots, \mu (\epsilon_{j-1}) \}$.
Since $[H (e_j)]$ also satisfies this property, we have $[H(e_j)] = \mu (\epsilon_j)$.

\begin{figure}[ht!]
\begin{center}
\begin{tabular}{cc}
\parbox[c]{4.8cm}{\includegraphics[width=4.4cm]{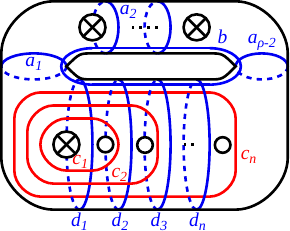}}
&
\parbox[c]{4.8cm}{\includegraphics[width=4.4cm]{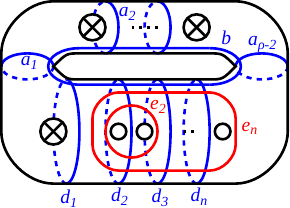}}
\end{tabular}

\caption{Circles in $N$ (Lemma \ref{lem3_14})}\label{fig3_6}
\end{center}
\end{figure}

We set $\gamma_j' = [H(c_j)]$ and we show by induction on $j$ that there exist $z_j,z_j' \in \Z \setminus \{ 0 \}$ such that $t_{\gamma_j}^{z_j} \in \GG$ and $\varphi (t_{\gamma_j}^{z_j}) = t_{\gamma_j'}^{z_j'}$.
Here and in the other proofs we use the well-known fact (which is a consequence of Lemma \ref{lem2_4} and Lemma \ref{lem2_5}) that, if an element $f \in \MM(N)$ commutes with a non-trivial power of a Dehn twist $t_\alpha$, then $f(\alpha) = \alpha$.
Suppose that $j=1$.
Let $u \in \Z \setminus \{ 0 \}$ such that $t_{\gamma_1}^u \in \GG$.
Set $f = \varphi(t_{\gamma_1}^u)$.
The element $t_{\gamma_1}^u$ commutes with every power of $t_\eta$ for all $\eta \in \CC= \{ \alpha_1, \dots, \alpha_{\rho-2}, \beta, \delta_2, \dots, \delta_n\}$.
By Proposition \ref{prop3_4} it follows that $f$ commutes with a non-trivial power of $t_{\eta'}$ for all $\eta' \in \mu (\CC)$, hence $f(\eta') = \eta'$ for all $\eta' \in \mu(\CC)$.
In particular each element of $\mu(\CC)$ is a reduction class for $f$.
None of them is an essential reduction class since $i (\mu (\alpha_i), \mu (\beta)) \neq 0$ for all $i \in \{1, \dots, \rho-2\}$ and $i (\mu (\beta), \mu (\delta_j)) \neq 0$ for all $j \in \{2, \dots, n\}$.
By Lemma \ref{lem3_13}, $f$ admits a representative $F: N \to N$ such that $F(H(a_1)) = H(a_1)$, $F(H(b)) = H(b)$, and $F(H(d_2)) = H(d_2)$.
Take a closed regular neighborhood $U$ of $H(a_1) \cup H(b) \cup H(d_2)$ such that $H(c_1)$ is a boundary component of $U$.
Observe that there is a unique connected component $K$ of $\overline{N \setminus U}$ homeomorphic to $N_{1,2}$.
We have $\partial K \cap \partial U = H(c_1)$, hence $F(H(c_1))$ is isotopic to $H(c_1)$, that is, $f (\gamma_1') = \gamma_1'$.
Let $\AA=\{ \mu (\alpha_1), \dots, \mu(\alpha_{\rho-2}), \mu (\delta_2), \dots, \mu(\delta_n), \gamma_1' \}$.
Note that $\AA$ is a simplex of $\TT (N)$ and $f(\eta') = \eta'$ for all $\eta' \in \AA$.
Let $\Lambda_\AA : \MM_\AA (N) \to \MM (N_\AA)$ be the reduction homomorphism along $\AA$.
Each connected component of $N_\AA$ is homeomorphic to $S_{0,3}$ or $N_{1,2}$ hence, by Lemma \ref{lem2_10}\,(1), $\MM (N_\AA)$ is finite.
It follows that there exists $v \in \Z \setminus \{ 0 \}$ such that $f^v \in Z_\AA = \Ker\,\Lambda_\AA$.
By the above the only element of $\AA$ which can be an essential reduction class for $f^v$ is $\gamma_1'$, hence, by Lemma \ref{lem2_5}, $f^v$ is a power of $t_{\gamma_1'}$.
We set $z_1 = uv$.
Then there exists $z_1' \in \Z \setminus \{ 0 \}$ such that $\varphi (t_{\gamma_1}^{z_1}) = t_{\gamma_1'}^{z_1'}$.

We assume that $j \ge 2$ plus the inductive hypothesis. 
Let $u \in \Z \setminus \{ 0 \}$ such that $t_{\gamma_j}^u \in \GG$.
Set $f = \varphi(t_{\gamma_j}^u)$.
The element $t_{\gamma_j}^u$ commutes with every power of $t_\eta$ for all $\eta \in \{ \alpha_1, \dots, \alpha_{\rho-2}, \beta, \delta_{j+1}, \dots,\delta_n, \epsilon_2, \dots, \epsilon_j, \gamma_1, \dots, \gamma_{j-1}\}$.
By Proposition \ref{prop3_4} and the inductive hypothesis it follows that $f$ commutes with a non-trivial power of $t_{\eta'}$ for all $\eta' \in \CC' = \{ \mu (\alpha_1), \dots, \mu(\alpha_{\rho-2}), \mu(\beta), \mu(\delta_{j+1}), \dots, \mu (\delta_n), \mu(\epsilon_2), \dots, \mu(\epsilon_j), \gamma_1', \dots, \gamma_{j-1}'\}$,
\break
hence $f(\eta') = \eta'$ for all $\eta' \in \CC'$.
In particular, each element of $\CC'$ is a reduction class for $f$.
None of them is an essential reduction class since $i (\mu (\alpha_i), \mu (\beta)) \neq 0$ for all $i \in \{1, \dots, \rho-2\}$, $i (\mu (\beta), \mu (\delta_k)) \neq 0$ for all $k \in \{j+1, \dots, n\}$, and $i (\mu(\epsilon_k), \gamma_{k-1}') \neq 0$ for all $k \in \{2, \dots, j\}$.
By Lemma \ref{lem3_13}, $f$ admits a representative $F: N \to N$ such that $F(H(a_1)) = H(a_1)$, $F(H(b)) = H(b)$, and $F(H(d_{j+1})) = H(d_{j+1})$.
Take a closed regular neighborhood $U$ of $H(a_1) \cup H(b) \cup H(d_{j+1})$ such that $H(c_j)$ is a boundary component of $U$.
Observe that there is a unique connected component $K$ of $\overline{N \setminus U}$ homeomorphic to $N_{1,j+1}$.
We have $\partial K \cap \partial U = H(c_j)$, hence $F(H(c_j))$ is isotopic to $H(c_j)$, that is, $f (\gamma_j') = \gamma_j'$.
Let $\AA=\{ \mu (\alpha_1), \dots, \mu(\alpha_{\rho-2}), \mu (\delta_{j+1}), \dots, \mu(\delta_n), \gamma_1', \dots, \gamma_{j-1}', \gamma_j' \}$.
Note that $\AA$ is a simplex of $\TT (N)$ and $f(\eta') = \eta'$ for all $\eta' \in \AA$.
Let $\Lambda_\AA : \MM_\AA (N) \to \MM (N_\AA)$ be the reduction homomorphism along $\AA$.
Every connected component of $N_\AA$ is homeomorphic to $S_{0,3}$ or $N_{1,2}$ hence, by Lemma \ref{lem2_10}\,(1), $\MM (N_\AA)$ is finite. 
It follows that there exists $v \in \Z \setminus \{ 0 \}$ such that $f^v \in Z_\AA = \Ker\,\Lambda_\AA$.
By the above the only element of $\AA$ which can be an essential reduction class for $f^v$ is $\gamma_j'$, hence, by Lemma \ref{lem2_5}, $f^v$ is a power of $t_{\gamma_j'}$.
We set $z_j = uv$.
Then there exists $z_j' \in \Z \setminus \{ 0 \}$ such that $\varphi (t_{\gamma_j}^{z_j}) = t_{\gamma_j'}^{z_j'}$.
\end{proof}

\begin{lem}\label{lem3_15}
Let $\gamma \in \TT_1 (N)$.
Assume that both connected components of $N_\gamma$ are (non-orientable) of genus $\ge 3$.
Then there exist $\gamma' \in \TT_1 (N)$ and $z,z' \in \Z \setminus \{0\}$ such that $t_\gamma^z \in \GG$ and $\varphi(t_\gamma^z) = t_{\gamma'}^{z'}$.
\end{lem}

\begin{proof}
Consider the configuration of circles in $N$ drawn in Figure \ref{fig3_7}.
We can and do choose this configuration so that $\gamma=[c]$.
We set $\alpha_i = [a_i]$ for $i \in \{1, \dots, \rho-2\}$, $\beta_i = [b_i]$ for $i \in \{1,2,3\}$, and $\delta_j = [d_j]$ for $j \in \{1, \dots, n\}$. 
By the results of Subsection \ref{subsec3_2} there exists a homeomorphism $H : N \to N$ such that $[H(u)] = \mu ([u])$ for all $u \in \{ a_1, \dots, a_{\rho-2}, b_1, b_2, b_3, d_1, \dots, d_{n} \}$.
We set $\gamma' = [H(c)]$ and we turn to show that there exist $z,z' \in \Z \setminus \{ 0 \}$ such that $t_{\gamma}^{z} \in \GG$ and $\varphi (t_{\gamma}^{z}) = t_{\gamma'}^{z'}$.
Let $u \in \Z \setminus \{ 0 \}$ such that $t_{\gamma}^u \in \GG$.
Set $f = \varphi(t_{\gamma}^u)$.
The element $t_{\gamma}^u$ commutes with every power of $t_\eta$ for all $\eta \in \CC= \{ \alpha_1, \dots, \alpha_{\rho-2}, \beta_1, \beta_3, \delta_1, \dots, \delta_{n}\}$.
By Proposition \ref{prop3_4} it follows that $f$ commutes with a non-trivial power of $t_{\eta'}$ for all $\eta' \in \mu (\CC)$, hence $f(\eta') = \eta'$ for all $\eta' \in \mu(\CC)$.
In particular each element of $\mu(\CC)$ is a reduction class for $f$.
None of them is an essential reduction class since $i (\mu (\alpha_i), \mu (\beta_1)) \neq 0$ for all $i \in \{1, \dots, p\}$, $i( \mu (\alpha_i), \mu (\beta_3)) \neq 0$ for all $i \in \{p+1, \dots, \rho-2\}$, $i (\mu (\delta_j), \mu (\beta_1)) \neq 0$ for all $j \in \{1, \dots, q\}$, and $i (\mu (\delta_j), \mu (\beta_3)) \neq 0$ for all $j \in \{q+1, \dots, n\}$. 
By Lemma \ref{lem3_13}, $f$ admits a representative $F: N \to N$ such that $F(H(a_i)) = H(a_i)$ for all $i \in \{1, \dots, p\}$, $F(H(b_1)) = H(b_1)$, and $F(H(d_j)) = H(d_j)$ for all $j \in \{1, \dots, q\}$. 
Take a closed regular neighborhood $U$ of $H(a_1) \cup \cdots \cup H(a_p)  \cup H(b_1) \cup H(d_1) \cup \cdots \cup H(d_q)$ such that $H(c)$ is a boundary component of $U$.
Observe that there is a unique connected component $K$ of $\overline{N \setminus U}$ non-orientable of genus $\ge 3$. 
We have $\partial K \cap \partial U = H(c)$, hence $F(H(c))$ is isotopic to $H(c)$, that is, $f (\gamma') = \gamma'$.
Let $\AA=\{ \mu (\alpha_1), \dots, \mu(\alpha_{\rho-2}), \mu (\delta_1), \dots, \mu(\delta_{n}), \gamma' \}$.
Note that $\AA$ is a simplex of $\TT (N)$ and $f(\eta') = \eta'$ for all $\eta' \in \AA$.
Let $\Lambda_\AA : \MM_\AA (N) \to \MM (N_\AA)$ be the reduction homomorphism along $\AA$.
Every connected component of $N_\AA$ is homeomorphic to $S_{0,3}$ or $N_{1,2}$ hence, by Lemma \ref{lem2_10}\,(1), $\MM (N_\AA)$ is finite.
It follows that there exists $v \in \Z \setminus \{ 0 \}$ such that $f^v \in Z_\AA = \Ker\,\Lambda_\AA$.
By the above the only element of $\AA$ which can be an essential reduction class for $f^v$ is $\gamma'$, hence, by Lemma \ref{lem2_5}, $f^v$ is a power of $t_{\gamma'}$.
We set $z = uv$.
Then there exists $z' \in \Z \setminus \{ 0 \}$ such that $\varphi (t_{\gamma}^{z}) = t_{\gamma'}^{z'}$.
\end{proof}

\begin{figure}[ht!]
\begin{center}
\includegraphics[width=7.2cm]{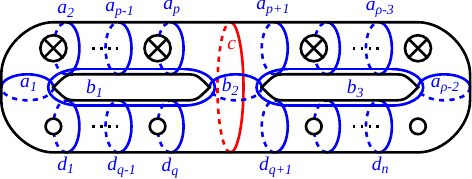}

\caption{Circles in $N$ (Lemma \ref{lem3_15})}\label{fig3_7}
\end{center}
\end{figure}

\begin{lem}\label{lem3_16}
Let $\gamma \in \TT_2 (N)$ and let $\gamma_0$ be the interior class of $\gamma$.
Then there exist $\gamma' \in \TT_2 (N)$, $z,z' \in \Z \setminus \{0\}$ and $z_0' \in \Z$ such that $\mu (\gamma_0)$ is the interior class of $\gamma'$,  $t_\gamma^z \in \GG$, and $\varphi(t_\gamma^z) = t_{\gamma'}^{z'} t_{\mu(\gamma_0)}^{z_0'}$.
\end{lem}

\begin{proof}
We suppose that the connected component of $N_\gamma$ non-homeomorphic to $N_{2,1}$ is non-orientable.
The case where this connected component is orientable can be proved in a similar way. 
We consider the configuration of circles in $N$ drawn in Figure \ref{fig3_8}.
We can and do choose this configuration so that $\gamma=[c]$.
We set $\alpha_i = [a_i]$ for $i \in \{1, \dots, \rho-1\}$, $\beta = [b]$, $\gamma_0 = [c_0]$, and $\delta_j = [d_j]$ for $j \in \{1, \dots, n-1\}$. 
By the results of Subsection \ref{subsec3_2}, there exists a homeomorphism $H : N \to N$ such that $[H(u)] = \mu ([u])$ for all $u \in \{ a_1, \dots, a_{\rho-1}, b, d_1, \dots, d_{n-1} \}$.
We set $\gamma' = [H(c)]$.
Note that $\gamma' \in \TT_2 (N)$ and $\gamma_0' = [H(c_0)]$ is the interior class of $\gamma'$.
It is easily seen that $\gamma_0'$ is the unique element of $\TT_0(N)$ that satisfies $i(\gamma_0', \eta) = 0$ for all $\eta \in \{ \mu (\alpha_1), \mu (\alpha_3), \dots, \mu(\alpha_{\rho-1}), \mu(\beta), \mu (\delta_1), \dots, \mu (\delta_{n-1})\}$.
Since $\mu(\gamma_0)$ also satisfies this property, we have $\mu (\gamma_0) = \gamma_0'$.
Let $u \in \Z \setminus \{ 0 \}$ such that $t_{\gamma}^u \in \GG$.
We set $f = \varphi(t_{\gamma}^u)$.
The element $t_{\gamma}^u$ commutes with every power of $t_\eta$ for all $\eta \in \CC= \{ \alpha_1, \alpha_3, \dots, \alpha_{\rho-1}, \beta, \delta_1, \dots, \delta_{n-1}\}$.
By Proposition \ref{prop3_4} it follows that $f$ commutes with a non-trivial power of $t_{\eta'}$ for all $\eta' \in \mu (\CC)$, hence $f(\eta') = \eta'$ for all $\eta' \in \mu(\CC)$.
In particular every element of $\mu(\CC)$ is a reduction class for $f$.
None of them is an essential reduction class since $i (\mu (\alpha_i), \mu (\beta)) \neq 0$ for all $i \in \{1, 3, \dots, \rho-1\}$ and $i (\mu (\delta_j), \mu (\beta)) \neq 0$ for all $j \in \{1, \dots, n-1\}$.
We have $f(\gamma_0') = \gamma_0'$ for the same reason (but we do not know if $\gamma_0'$ is an essential reduction class for $f$).
By Lemma \ref{lem3_13}, $f$ admits a representative $F: N \to N$ such that $F(H(a_i)) = H(a_i)$ for all $i \in \{1,3, \dots, \rho-1\}$, $F(H(b)) = H(b)$, and $F(H(d_j)) = H(d_j)$ for all $j \in \{1, \dots, n-1\}$. 
Take a closed regular neighborhood $U$ of $H(a_1) \cup H(a_3) \cup \cdots \cup H(a_{\rho-1}) \cup H(b) \cup H(d_1) \cup \cdots \cup H(d_{n-1})$ such that $H(c)$ is a boundary component of $U$.
Observe that there is a unique connected component $K$ of $\overline{N \setminus U}$ homeomorphic to $N_{2,1}$, and the boundary component of $K$ is $H(c)$.
So, $F(H(c))$ is isotopic to $H(c)$, that is, $f (\gamma') = \gamma'$.
Let $\AA=\{ \mu (\alpha_1), \mu (\alpha_3), \dots, \mu(\alpha_{\rho-1}), \mu (\delta_1), \dots, \mu(\delta_{n-1}), \gamma_0', \gamma' \}$.
Note that $\AA$ is a simplex of $\TT (N)$ and $f(\eta') = \eta'$ for all $\eta' \in \AA$.
Let $\Lambda_\AA : \MM_\AA (N) \to \MM (N_\AA)$ be the reduction homomorphism along $\AA$.
All the connected components of $N_\AA$ are homeomorphic to $S_{0,3}$ or $N_{1,2}$ hence, by Lemma \ref{lem2_10}\,(1), $\MM (N_\AA)$ is finite. 
It follows that there exists $v \in \Z \setminus \{ 0 \}$ such that $f^v \in Z_\AA = \Ker\,\Lambda_\AA$.
By the above the only elements of $\AA$ that can be essential reduction classes for $f^v$ are $\gamma'$ and $\gamma_0'$, hence, by Lemma \ref{lem2_5}, $f^v \in \langle t_{\gamma'}, t_{\gamma_0'} \rangle$. 
We set $z = uv$.
Then there exist $z', z_0' \in \Z$ such that $\varphi (t_{\gamma}^{z}) = t_{\gamma'}^{z'} t_{\gamma_0'}^{z_0'}$.

\begin{figure}[ht!]
\begin{center}
\bigskip
\includegraphics[width=5cm]{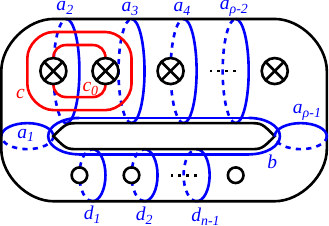}

\caption{Circles in $N$ (Lemma \ref{lem3_16})}\label{fig3_8}
\end{center}
\end{figure}

It remains to show that $z' \neq 0$.
Suppose that $z' =0$.
Then $z_0' \neq 0$ since $t_{\gamma}^z \neq 1$.
By Proposition \ref{prop3_4} there exist $y,y' \in \Z \setminus \{ 0 \}$ such that $t_{\gamma_0}^y \in \GG$ and $\varphi (t_{\gamma_0}^y) = t_{\gamma_0'}^{y'}$.
We have $\varphi (t_\gamma^{zy'}) = t_{\gamma_0'}^{z_0'y'} = \varphi (t_{\gamma_0}^{z_0'y})$, hence $t_\gamma^{zy'} = t_{\gamma_0}^{z_0'y}$, and therefore $\gamma = \gamma_0$: contradiction.
So, $z' \neq 0$.
\end{proof}

\begin{proof}[Proof of Theorem \ref{thm3_1}]
Let $\alpha \in \TT (N)$.
If $\alpha \in \TT_0 (N)$, then we set $\lambda (\alpha) = \mu (\alpha)$.
Suppose that $\alpha \in \TT_1 (N)$.
By Lemma \ref{lem3_14} and Lemma \ref{lem3_15} there exist $\alpha' \in \TT_1 (N)$ and $x,x' \in \Z \setminus \{ 0 \}$ such that $t_\alpha^x \in \GG$ and $\varphi( t_\alpha^x) = t_{\alpha'}^{x'}$.
We choose such a $\alpha'$ and we set $\lambda(\alpha) = \alpha'$.
Suppose that $\alpha \in \TT_2 (N)$.
Let $\alpha_0$ be the interior class of $\alpha$.
By Lemma \ref{lem3_16} there exist $\alpha' \in \TT_2 (N)$, $x,x' \in \Z \setminus \{0\}$, and $y' \in \Z$, such that $t_\alpha^x \in \GG$, $\mu(\alpha_0)$ is the interior class of $\alpha'$, and $\varphi (t_\alpha^x) = t_{\alpha'}^{x'} t_{\mu (\alpha_0)}^{y'}$.
We choose such a $\alpha'$ and we set $\lambda (\alpha) = \alpha'$.

We need to prove that $\lambda : \TT (N) \to \TT (N)$ is a super-injective simplicial map. 
To do that we take two distinct classes $\alpha, \beta \in \TT  (N)$ and we show that $i(\alpha, \beta) = 0$ if and only if $i( \lambda (\alpha), \lambda (\beta)) = 0$.
It is easily shown using the same arguments as in the proof of Proposition \ref{prop3_4} that this is true if $\alpha, \beta \in \TT_0 (N) \cup \TT_1 (N)$.  
To prove the other cases we will use the following claims.

{\it Claim 1.}
Let $\gamma \in \TT_2 (N)$, let $\gamma_0$ be the interior class of $\gamma$, and let $\delta \in \TT (N)$.
If $i(\delta, \gamma) = 0$, then $i (\delta, \gamma_0) = 0$.

{\it Proof of Claim 1.}
Suppose that $i(\delta, \gamma)=0$ and $i(\delta, \gamma_0) \neq 0$.
Let $N_1$ be the connected component of $N_\gamma$ homeomorphic to $N_{2,1}$.
Since $i(\delta, \gamma)=0$ and $i(\delta, \gamma_0) \neq 0$, we should have $\delta \in \TT (N_1) = \{ \gamma_0 \}$, hence $\delta = \gamma_0$, and therefore $i(\delta, \gamma_0)=0$: contradiction.
This proves Claim 1.

{\it Claim 2.}
Let $\gamma \in \TT_2 (N)$ and let $\gamma_0$ be the interior class of $\gamma$.
There exist $z,z' \in \Z \setminus \{0\}$ and $q \in \Z$ such that $t_\gamma^z t_{\gamma_0}^q\in \GG$ and $\varphi(t_\gamma^z t_{\gamma_0}^q) = t_{\lambda(\gamma)}^{z'}$.

{\it Proof of Claim 2.}
Let $x,x' \in \Z \setminus \{ 0 \}$ and $p' \in \Z$ such that $\varphi (t_\gamma^x) = t_{\lambda(\gamma)}^{x'} t_{\lambda(\gamma_0)}^{p'}$.
Let $y,y' \in \Z \setminus \{ 0 \}$ such that $t_{\gamma_0}^y \in \GG$ and $\varphi (t_{\gamma_0}^y) = t_{\lambda (\gamma_0)}^{y'}$.
Then $t_{\gamma}^{xy'} t_{\gamma_0}^{-p' y} \in \GG$ and $\varphi (t_{\gamma}^{xy'} t_{\gamma_0}^{-p' y}) = t_{\lambda (\gamma)}^{x'y'}$.
This proves Claim 2. 

Suppose that $\alpha \in \TT_0 (N) \cup  \TT_1 (N)$ and $\beta \in \TT_2 (N)$.
Let $\beta_0$ be the interior class of $\beta$.
We choose $x,x',y,y',z,z' \in \Z \setminus \{ 0 \}$ and $p',q \in \Z$ such that $t_\alpha^x, t_\beta^y, t_\beta^z t_{\beta_0}^q \in \GG$, $\varphi( t_\alpha^x) = t_{\lambda (\alpha)}^{x'}$,  $\varphi (t_\beta^y) = t_{\lambda (\beta)}^{y'} t_{\lambda (\beta_0)}^{p'}$, and $\varphi (t_\beta^z t_{\beta_0}^q) = t_{\lambda (\beta)}^{z'}$.
Suppose that $i( \lambda (\alpha), \lambda (\beta)) = 0$.
By Claim 1 we have $i(\lambda(\alpha), \lambda (\beta_0)) = 0$.
Then $\varphi (t_\alpha^x) = t_{\lambda (\alpha)}^{x'}$ and $\varphi( t_\beta^y) = t_{\lambda(\beta)}^{y'} t_{\lambda (\beta_0)}^{p'}$ commute, hence $t_\alpha^x$ and $t_\beta^y$ commute, and therefore $i(\alpha, \beta)=0$.
Suppose that $i(\alpha, \beta) = 0$.
By Claim 1 we have $i(\alpha, \beta_0) = 0$.
Then $t_\alpha^x$ and $t_\beta^z t_{\beta_0}^q$ commute, hence $\varphi (t_\alpha^x) = t_{\lambda (\alpha)}^{x'}$ and $\varphi (t_\beta^z t_{\beta_0}^q) = t_{\lambda (\beta)}^{z'}$ commute, and therefore $i(\lambda (\alpha), \lambda (\beta))=0$.

Suppose that $\alpha, \beta \in \TT_2 (N)$.
Let $\alpha_0$ be the interior class of $\alpha$ and let $\beta_0$ be the interior class of $\beta$.
Let $x,x',y,y',z,z',u,u' \in \Z \setminus \{ 0 \}$ and $p',q,r',s \in \Z$ such that $t_\alpha^x, t_\alpha^y t_{\alpha_0}^q, t_\beta^z, t_\beta^u t_{\beta_0}^s \in \GG$, $\varphi (t_\alpha^x) = t_{\lambda (\alpha)}^{x'} t_{\lambda (\alpha_0)}^{p'}$, $\varphi (t_\alpha^y t_{\alpha_0}^q) = t_{\lambda (\alpha)}^{y'}$, $\varphi (t_\beta^z) = t_{\lambda (\beta)}^{z'} t_{\lambda (\beta_0)}^{r'}$, and $\varphi (t_\beta^u t_{\beta_0}^s) = t_{\lambda (\beta)}^{u'}$.
Suppose that $i( \lambda (\alpha), \lambda (\beta)) = 0$.
By Claim 1 we have $i (\lambda (\alpha_0), \lambda (\beta)) = i (\lambda (\alpha), \lambda (\beta_0)) = 0$, hence, again by Claim 1, $i (\lambda (\alpha_0), \lambda (\beta_0))=0$.
Then $\varphi (t_\alpha^x) = t_{\lambda (\alpha)}^{x'} t_{\lambda (\alpha_0)}^{p'}$ and $\varphi (t_\beta^z) = t_{\lambda (\beta)}^{z'} t_{\lambda (\beta_0)}^{r'}$ commute, hence $t_\alpha^x$ and $t_\beta^z$ commute, and therefore $i (\alpha, \beta)=0$.
Suppose that $i (\alpha, \beta) = 0$.
By Claim 1 we have $i(\alpha_0, \beta) = i(\alpha, \beta_0)=0$, hence, again by Claim 1, $i (\alpha_0, \beta_0) =0$.
Then $t_\alpha^y t_{\alpha_0}^q$ and $t_\beta^u t_{\beta_0}^s$ commute, hence $t_{\lambda (\alpha)}^{y'} = \varphi (t_{\alpha}^y t_{\alpha_0}^q)$ and $t_{\lambda (\beta)}^{u'} = \varphi (t_\beta^u t_{\beta_0}^s)$ commute, and therefore $i (\lambda (\alpha), \lambda (\beta)) = 0$.
\end{proof}


\section{Proofs of Theorem \ref{thm1_1} and Corollary \ref{corl1_3}}\label{Sec4}

\begin{proof}[Proof of Theorem \ref{thm1_1}]
Let $\GG$ be a finite index subgroup of $\MM (N)$ and let $\varphi : \GG \to \MM (N)$ be an injective homomorphism. 
We denote by $\lambda : \TT (N) \to \TT (N)$ the super-injective simplicial map induced by $\varphi$ (see Theorem \ref{thm3_1}).
We know by Irmak--Paris \cite[Theorem 1.1]{IrmPar1} that there exists $f_0 \in \MM (N)$ such that $\lambda (\alpha) = f_0 (\alpha)$ for all $\alpha \in \TT (N)$.
We turn to show that $\varphi (g) = f_0 g f_0^{-1}$ for all $g \in \GG$.

We say that a collection $\CC \subset \TT (N)$ has \emph{trivial stabilizer} if the only element $f$ in $\MM (N)$ that satisfies $f(\alpha) = \alpha$ for all $\alpha \in \CC$ is $f=\id$.
It is shown in Irmak--Paris \cite[Lemma 3.6, Lemma 3.7]{IrmPar1} that there exists a finite collection $\CC \subset \TT_0 (N)$ which has trivial stabilizer. 
Take such a collection $\CC$.
Let $g \in \GG$.
Let $\alpha \in \CC$.
Since $\alpha, g(\alpha) \in \TT_0 (N)$, by Theorem \ref{thm3_1}, there exist $x,x',y,y' \in \Z \setminus \{ 0 \}$ such that $t_\alpha^x, t_{g(\alpha)}^y \in \GG$, $\varphi (t_\alpha^x) = t_{\lambda (\alpha)}^{x'}$, and $\varphi (t_{g(\alpha)}^y) = t_{\lambda (g(\alpha))}^{y'}$.
We have 
\begin{gather*}
t_{(\varphi(g)f_0)(\alpha)}^{x'y} 
= \varphi(g) t_{f_0 (\alpha)}^{x'y} \varphi(g)^{-1} 
= \varphi(g) t_{\lambda (\alpha)}^{x'y} \varphi(g)^{-1} 
= \varphi(g) \varphi (t_\alpha^{xy}) \varphi(g)^{-1}\\
= \varphi(g t_\alpha^{xy}g^{-1})
= \varphi(t_{g(\alpha)}^{xy})
= t_{\lambda(g(\alpha))}^{xy'}
= t_{(f_0g)(\alpha)}^{xy'}\,.
\end{gather*}
So, $(\varphi(g)f_0)(\alpha) = (f_0g)(\alpha)$, that is, $(g^{-1} f_0^{-1} \varphi(g) f_0)(\alpha) = \alpha$.
Since $\CC$ has trivial stabilizer, it follows that $g^{-1} f_0^{-1} \varphi (g) f_0 = \id$, hence $\varphi (g) = f_0 g f_0^{-1}$.
\end{proof}

\begin{proof}[Proof of Corollary \ref{corl1_3}]
For $f \in \MM (N)$ we denote by $c_f : \MM (N) \to \MM (N)$, $g \mapsto fgf^{-1}$, the conjugation by $f$. 
We have an obvious homomorphism $\Psi: \MM (N) \to \Com (\MM (N))$ which sends $f$ to the class of $(\MM(N), \MM(N),c_f)$.
Theorem \ref{thm1_1} clearly implies that $\Psi$ is surjective, and the following lemma implies that $\Psi$ is injective. 
\end{proof}

\begin{lem}\label{lem4_1}
Let $\GG$ be a finite index subgroup of $\MM (N)$.
Then the centralizer of $\GG$ in $\MM (N)$ is trivial.
\end{lem}

\begin{proof}
As in the proof of Theorem \ref{thm1_1}, we may choose a finite collection $\CC \subset \TT_0 (N)$ having trivial stabilizer. 
Let $g$ be an element in the centralizer of $\GG$.
Let $\alpha \in \CC$.
Take $x \in \Z \setminus \{ 0 \}$ such that $t_\alpha^x \in \GG$.
Then $t_\alpha^x = g t_\alpha^x g^{-1} = t_{g(\alpha)}^x$, hence $g(\alpha) = \alpha$.
Since $\CC$ has trivial stabilizer, we conclude that $g = \id$.
\end{proof}



\end{document}